\documentclass[12pt, a4paper, reqno, twoside]{amsart}

\usepackage[utf8]{inputenc}
\usepackage[T1]{fontenc}

\usepackage[UKenglish]{babel}

\usepackage{amsmath,amscd,amsfonts,amssymb,graphicx,color}
\usepackage{mathtools}
\usepackage{xspace}
\usepackage{xfrac}
\usepackage{faktor}
\usepackage{accents}
\usepackage{relsize}
\usepackage{nameref}

\usepackage[%
]{hyperref}
\hypersetup{
  colorlinks  = true,
  allcolors  = blue,  
}

\usepackage[p]{stickstootext}
\usepackage[stix2,vvarbb,smallerops]{newtxmath}

\usepackage{esstixfrak}
\usepackage[scr=rsfso,scrscaled=1.03]{mathalpha}

\DeclareTextCommandDefault{\cprime}{$'$}

\numberwithin{equation}{section}

\newenvironment{mypmatrix}[1][1.3]{%
\renewcommand*{\arraystretch}{#1}%
\setlength{\arraycolsep}{3pt}%
\begin{pmatrix}}{\end{pmatrix}}

\DeclareMathOperator{\Max}{Max}

\newcommand{\frakS}{\mathfrak S}
\newcommand{\frakL}{\mathfrak L}
\newcommand{\frakM}{\mathfrak M}
\newcommand{\frakN}{\mathfrak N}

\newcommand{\scrH}{\mathscr H}
\newcommand{\scrF}{\mathscr F}

\newcommand{\ie}{\mbox{i.\,e.}\xspace}
\newcommand{\eg}{\mbox{e.\,g.}\xspace}
\newcommand{\cf}{\mbox{cf.}\xspace}

\newcommand{\MYQED}{$\square$}

\newtheorem{theorem}{Theorem}[section]
\newtheorem{lemma}[theorem]{Lemma}

\newtheorem{corollary}[theorem]{Corollary}

\theoremstyle{definition}

\theoremstyle{remark}
\newtheorem{example}[theorem]{Example}
\newtheorem{remark}[theorem]{Remark}
\newtheorem*{example*}{}
\newtheorem*{problem*}{Problem}

\newtheoremstyle{mystyle}
  {}                % ABOVESPACE
  {}                % BELOWSPACE
  {}                % BODYFONT
  {0pt}             % INDENT
  {\bfseries}       % HEADFONT
  {}                % HEADPUNCT
  {1ex}             % HEADSPACE
  {\thmnote{#3}}    % CUSTOM-HEAD-SPEC

\theoremstyle{mystyle}  
\newtheorem*{fact*}{}

\AtBeginDocument{%
   \def\MR#1{}
 }

\begin{document}

\setcounter{page}{1}

\title[Remarks on infimum and maximal lower bounds of a set of 
operators]{Remarks on infimum and maximal lower bounds\\
of a set of bounded self-adjoint operators}

\author[M.~G\"unther \MakeLowercase{and} L.~Klotz]%
{Matthias~G\"unther$^1$$^{\ast}$ \MakeLowercase{and} Lutz~Klotz$^{2}$}

\address{\hspace*{0pt}\llap{$^{1}$\,}%
Mathematisches Institut, Universit\"at Leipzig, 
PF 10\,09\,20, 04109 Leipzig, Ger\-many.}

\email{\href{mailto:guenther@math.uni-leipzig.de}{guenther@math.uni-leipzig.de}}

\address{\hspace*{0pt}\llap{$^{2}$\,}%
Wiesenstr.~2, 08309 Eibenstock, OT Sosa, Germany.}

\email{\href{mailto:lutzklotz@t-online.de}{lutzklotz@t-online.de}}



\date{\today.\newline\indent$^{*}$\,Corresponding author}

\begin{abstract}
The notions of infimum and maximal lower bounds of a set $\frakM$ of
bounded self-adjoint operators were mainly studied for a set $\frakM$
of two elements. The present paper deals with more general sets
$\frakM$, where it is required that $\frakM$ is nonempty and bounded
from below. Kadison's theorem on the existence of the infimum of a
two-element set is proved for a countable and weak-operator compact
set $\frakM$. Stott's recent results on the structure of the set of
maximal lower bounds of a finite set of Hermitian matrices are
discussed and partially generalized. We are also concerned with
the greatest lower bound and maximal lower bounds under certain
restrictions. It is shown that the set of all lower bounds of $\frakM$
commuting with all elements of $\frakM$ possesses the greatest element if
$\frakM$ is a set of pairwise commuting operators. The theorem of
Moreland and Gudder on the existence of the greatest positive lower
bound of a set of two positive matrices is extended to an arbitrary
finite set of positive matrices.
\end{abstract}

\maketitle

% keine Seitenzahl auf erster Seite
\thispagestyle{empty} 


% ------------------------------------------------------------------
% Beginn der Abschnitte
% ------------------------------------------------------------------

\section{Introduction}
\label{sec1}
Let $\scrH$ be a complex Hilbert space with inner product
$(\cdot,\cdot)$ and corresponding norm $\|\cdot\|$. The dimension of
$\scrH$ is denoted by $\dim\scrH$ and $\overline{\scrF}$ stands for the
closure of a subset $\scrF$ of $\scrH$. Let $\frakS(\scrH) = \frakS$
denote the $\mathbb R$-linear space of bounded self-adjoint operators
on $\scrH$, whose norm is the usual operator norm. The space $\frakS$
is endowed with L\"owner's partial order, \ie for $S,T\in\frakS$, we
write $S\leq T$ or $T\geq S$ if and only if $(Su,u)\leq(Tu,u)$, $u\in
\scrH$. Two operators $S$ and $T$ are called comparable if $S\leq T$
or $T\leq S$. An operator $S\in\frakS$ is called positive if $S\geq
0$,  where here and throughout $0$ stands for an appropriate zero
element.  Similarly, the symbol $I$ stands for an identity operator,
whose  domain should be clear from the context. We write $\mathscr
R(S)$  and $\mathscr N(S)$ for the range and the null space, resp., of
$S\in\frakS$. The cone of all positive operators is denoted by
$\frakS^{+}$  and $S^{\sfrac{1}{2}}$ denotes the unique positive
square root of an operator $S\in\frakS^{+}$. If $S\in\frakS$ we set
$|S|:=(S^{2})^{\sfrac{1}{2}}$. Recall that for $S\in\frakS^{+}$, a
vector $u$ belongs to $\mathscr N(S)$ if and only if $(Su,u)=0$.
The notation `projection' always means `orthogonal projection'.

Let $\frakM$ be a nonempty subset of $\frakS$ and
$\frakL=\frakL(\frakM)$ be the set of its lower bounds with respect to
L\"owner's partial order. In the theory of operator algebras, the
existence of the greatest element of $\frakL$ is of considerable interest.
It is common to call such an element the infimum of $\frakM$, and we
denote it by $\inf\frakM$. A first important and, perhaps, somewhat
surprising result was obtained by Kadison~\cite{Kadison1951},
\cf \cite[Exercise~2.8.18]{KadisonRingrose1983}, who proved that a
set  $\frakM:=\{A,B\}$ of two operators possesses the infimum only in the
trivial case of comparable operators $A$ and $B$. This means that the
order structure generated by L\"owner's partial order is as far as
possible from a lattice structure, and Kadison called it an anti-lattice.

Although the infimum of a set $\frakM$ need not to exist, Zorn's lemma
guarantees that the set of maximal elements of $\frakL$ is not empty
if $\frakL$ itself is not empty. In \cite{Stott2016arxiv} and his
dissertation \cite[Sec.~2.3]{Stott2017dis} Stott gave a complete
description  of the  set of maximal lower bounds of a set $\frakL$ of
two  Hermitian matrices,  \cf also \cite{GaubertStott2017}. To be
precise, having in mind applications, Stott was concerned with real
matrices but noticed that a large part of his results can be
extended to the complex case easily. With the aid of his result, he
derived Kadison's theorem in the case of matrices and at the same time
pointed out possible applications to reachability analysis of
dynamical  systems and program verification, information geometry, and
mathematical morphology. Moreover, Stott gave some results on the
structure  of the set of maximal lower bounds of an arbitrary finite
set  $\frakM$ of Hermitian matrices, \cf \cite[Sect.~2.6]{Stott2017dis}, \cite[Thm.~1]{GaubertStott2017}.

Along with the infimum of $\frakM$ the existence of the greatest element
of a certain subset of $\frakL(\frakM)$ has gained interest. We
mention several assertions of this type.

\begin{theorem}
\label{thm1.1}
If $\frakM$ is a set of two projections $P$ and $Q$, the
projection onto $\mathscr R(P)\cap\mathscr R(Q)$ is the greatest
projection not exceeding both $P$ and $Q$.
\end{theorem}  

\begin{theorem}
\label{thm1.2}
Let $\frakM:=\{A,B\}$ be a set of two commuting operators and
$\frakL^{com}$ be the set of all lower bounds of $\frakM$ commuting
with $A$ and $B$. The operator
\begin{equation}
\label{eq1.1}
M:=\frac{1}{2}\bigl(A+B-|A-B|\bigr)
\end{equation}
is the greatest element of $\frakL^{com}$.
\end{theorem}  

If $A,B\in\frakS^{+}$, let $A:B$ denote its parallel sum,
\cf \cite{Anderson1971} or \cite{PekarevSmuljan1976} for the definition
and properties of $A:B$. Let $[B]A:=\lim_{m\to\infty}A:mB$, where the
limit exists with respect to the strong-operator topology, \cf
\cite{Ando1999}.

\begin{theorem}
\label{thm1.3}
Let $A,B\in\frakS^{+}$, $\frakM=\{A,B\}$, and $\frakL^{+}$ be the set
of all positive lower bounds of $\frakM$. The set $\frakL^{+}$ has the
greatest element if and only if $[A]B$ and $[B]A$ are comparable. In
this case, it is equal to $\inf\bigl\{[A]B,[B]A\bigr\}$.
\end{theorem}

The assertion of Theorem~\ref{thm1.1} shows that the set of
projections is a lattice with respect to L\"owner's partial
order. Theorem~\ref{thm1.2} can be found in \cite[Nr.~108]{RieszNagy1973}
and it was used there to prove the spectral theorem for bounded
self-adjoint operators. Theorem~\ref{thm1.3} is the culmination of a series
of papers dealing with quantum effects. First Moreland and Gudder
\cite{MorelandGudder1999} derived a necessary and sufficient condition
for  the existence of the greatest element of $\frakL^{+}$, if  $\frakM$ is
a set  of two positive matrices and asked for an extension of their
criterion to operators in an infinite-dimensional Hilbert space.
Gheondea, Gudder and Jonas \cite{GheondeaGudderJonas2005}
showed by an example that the condition given in \cite{MorelandGudder1999} 
is not necessary if $\dim\mathscr H=\infty$. The assertion of
Theorem~\ref{thm1.3} was proved by Ando~\cite{Ando1999}. It is
equivalent to Moreland's and Gudder's theorem if $\dim\mathscr
H<\infty$, \cf \cite{Ando2005}. Moreover, Stott
\cite[Cor.~3.5]{Stott2016arxiv} described the set of all maximal
elements of  $\frakL^{+}$ for a set $\frakM$ of two positive matrices
and this way could give an alternative proof of the theorem of
\cite{MorelandGudder1999}. Partial results were also stated in 
\cite{YuanDu2006}, \cf \cite{Gheondea2007} for short proofs.
It is also worth mentioning that Tarcsay and G\"ode
\cite{TarcsayGoede2024} extended Ando's theorem from a Hilbert space
setting to an anti-dual pair of spaces.

As can be seen from the results listed above, most authors have
studied lower bounds of a set $\frakM$ of two operators. Unlike these
papers, the main goal of the present paper is a discussion of the set
of lower bounds of a more general set $\frakM$. More precisely, we
have tried to generalize the known results for a set $\frakM$ of two
elements. We emphasize that even a generalization from a set $\frakM$ 
of two elements to an arbitrary finite set is not as straightforward as
one might expect. One reason for this is the fact, that \eg by an
application of \eqref{eq2.3} below, for a set $\frakM=\{A,B\}$ of two
operators it can often be supposed that $A$ and $B$ commute, whereas a
condition of pairwise commuting operators of a larger set $\frakM$ is
a rather serious restriction. Compare \cite{Ando1999}, where the
transition to commuting operators $A$ and $B$ was justified in another way.
Also, Stott \cite{Stott2016arxiv, Stott2017dis} succeeded in
applying the indefinite metric generated by the matrix $A-B$, an
approach, which hardly could be adapted to a set of more than two
matrices. Moreover, if the set $\frakM$ is infinite, topological
properties of $\frakM$ become important. Since we feel that the
subject deserves further investigation, we have formulated a number
of questions and open problems.

Section~\ref{sec2} summarizes simple properties of the greatest lower
bound  and the set of maximal lower bounds. Some of them are
frequently  used in the sequel, others were mentioned since we find
them  of interest in their own. At the end of Section~\ref{sec2}
we recall the notion of a generalized Schur complement and its
strongly related concept of a shorted operator. We refer to
\cite{Shmulyan1959}, where many aspects of the subject and their
relations are discussed in detail as well as to \cite{Anderson1971,
AndersonTrapp1975} for more compact presentations of the facts. We
state an  operator version of the well known Albert's positivity
criterion for $2\times 2$ block matrices \cite{Albert1969}. This
assertion plays an essential role in our considerations since,
again by \eqref{eq2.3} below, it can often be assumed without loss of 
generality that $\frakM$ is a subset of $\frakS^{+}$.

Section~\ref{sec3} is devoted to a generalization of Kadison's
theorem.  We extend it to an arbitrary countable and weak-operator
compact set and show by counterexamples that it does not remain true
if  one of the conditions on the set $\frakM$ is omitted.

Section~\ref{sec4} deals with the set of maximal lower bounds. We give a
sufficient condition for a lower bound of $\frakM$ to be maximal and
prove that it is also necessary if $\frakM$ is finite. If
$\dim\mathscr H <\infty$, our result coincides with the equivalence
(i) $\Leftrightarrow$ (ii) of Theorem~1 of \cite{GaubertStott2017}. We
discuss the existence of a maximal lower bound, which is equal to one
of the operators of $\frakM$ on a prescribed vector and obtain
generalizations of Stott's results concerning the set $\frakM$ of two
matrices \cite{Stott2016arxiv}.
% We also propose an algorithm to
% compute a positive maximal lower bound of a set of positive matrices.

The short Section~\ref{sec5} contains some additional comments on Stott's
paper. In particular, we give an alternative proof of one of Stott's
main results, which points out the role of the indefinite metric
generated by the difference of the two matrices of $\frakM$.

Since for projections $P$ and $Q$, the inequality $P\leq Q$ is
equivalent to the inclusion $\mathscr R(P)\subseteq\mathscr R(Q)$,
\cf \cite[Prop.~2.5.2]{KadisonRingrose1983}, Theorem~\ref{thm1.1} can
be immediately extended from a set of two projections to an
arbitrary  set of projections. More interesting are generalizations of
Theorems~\ref{thm1.2} and \ref{thm1.3}. Some of them are derived in
Section~\ref{sec6}. We extend Theorem~\ref{thm1.2} to an arbitrary set
$\frakM$ of pairwise commuting operators. Under the condition
$\dim\mathscr H<\infty$, an amplification of Theorem~\ref{thm1.3} to
an arbitrary finite  set is given, a generalization of Moreland's and
Gudder's theorem \cite{MorelandGudder1999}.

\section{Elementary properties of greatest lower bound 
and maximal lower bounds}
\label{sec2}

Let $\frakL$ be a set, which is partially ordered with respect to an
order relation denoted by $\leq$. An element $G$ of $\frakL$ is called
the \emph{greatest} element of $\frakL$ if and only if $L\leq G$ for all
$L\in\frakL$. An element $M$ of $\frakL$ is called \emph{maximal}
element if and only if $L\in\frakL$ and $M\leq L$ imply that $M=L$. It
is obvious that neither the greatest element nor a maximal element need
to exist and that the greatest element is unique if it exists. The
greatest element of $\frakL$ will be denoted by $G:=G(\frakL)$ and the
set of maximal elements by $\Max\frakL$.

Let $\mathscr H$ be a Hilbert space over $\mathbb C$ and $\frakS$ be
the space of bounded self-adjoint linear operators on $\mathscr H$,
endowed with L\"owner's partial order.

Let us state an important result, \cf \cite[Thm.~3.2]{Shmulyan1959},
which is useful in the integration theory with respect to a
matrix-valued measure \cite{Shmulyan1959} and in multivariate
prediction theory \cite{Klotz2002}. It was also mentioned in
\cite[Appendix II]{Dixmier1981} with reference to an earlier
dissertation by Vigier.

\begin{theorem}
\label{thm2.1a}
Let $\frakM$ be a nonempty subset of $\frakS$ such that $A\leq S$ for
some $S\in\frakS$ and all $A\in\frakM$. If for all $A,B\in\frakM$
there exists $C\in\frakM$ with $A\leq C$ and $B\leq C$, the greatest
element $G(\frakM)$ of $\frakM$ exists and
\begin{equation}
\label{eq2.1a}
\bigl(G(\frakM)u,u\bigr)=\sup\bigl\{(Au,u) : A\in\frakM\bigr\},
\; u\in\mathscr H.
\end{equation}
\end{theorem}

Let $\frakM$ be a nonempty subset of $\frakS$. An operator
$L\in\frakS$ is called a \emph{lower bound} of $\frakM$ if $L\leq A$
for all $A\in\frakM$. The set of all lower bounds of $\frakM$ is
denoted  by $\frakL:=\frakL(\frakM)$. If the set $\frakL$ is not
empty,  the set $\frakM$ is called \emph{bounded from below}.
Throughout the rest of the paper, \emph{the symbol $\frakM$
stands for a nonempty subset of $\frakS$, which is bounded from
below}. If we speak of $\frakM$ as a set of matrices, this will mean
that the underlying Hilbert space has finite dimension. It is common
in literature to call the greatest lower bound of $\frakM$ the
\emph{infimum} of $\frakM$, and we shall set $\inf\frakM:=G(\frakL)$.

\begin{theorem}
\label{thm2.1}
For arbitrary $L\in\frakL$, there exists $M\in\Max\frakL$ with $L\leq
M$. In particular, the set $\Max\frakL$ is not empty. It is a
singleton  if and only if\/ $\inf\frakM$ exists. In this case, $\Max\frakL=\{\inf\frakM\}$.
\end{theorem}

\begin{proof}
For $L\in\frakL$, set $\widetilde{\frakL}:=\{S\in\frakL : S\geq L\}$.
Applying Theorem~\ref{thm2.1a} to a totally ordered subset of
$\widetilde{\frakL}$, we obtain that $\widetilde{\frakL}$ satisfies
the conditions of Zorn's lemma and the existence of
$M\in\Max\widetilde{\frakL}$ with $L\leq M$ follows. Clearly, $M$ is a
maximal element of $\frakL$ as well, which proves the first assertion.
The second assertion is obvious. Finally, let $\Max\frakL=\{M\}$ be a
singleton. If $L\in\frakL$, according to the first assertion just
proved  there exists $\widetilde{M}\in\Max\frakL$ with
$L\leq\widetilde{M}$.  Since $M$ is the only maximal element of
$\frakL$, it follows $M=\widetilde{M}\geq L$, $L\in\frakL$, which
yields $M=\inf\frakM$. 
\end{proof}

Let us summarize the following simple facts
\nameref{fa2.a}-\nameref{fa2.g}  on the infimum and the set of maximal
lower bounds of a set $\frakM$.

%\begin{enumerate}[align=left,
%  labelindent= 0pt,
%  labelwidth = 1em,
%  labelsep   = 0.5em,
%  itemindent = 1.5em,
%  leftmargin = 0pt,
%  parsep     = 0.5ex
%  ]
  
% \begin{itemize}[align=left,
%   left = 0pt,
%  leftmargin = *
%  ]

\begin{fact*}[(2.a)]
\label{fa2.a} 
If $\frakN$ is a nonempty subset of $\frakM$, clearly
$\frakL(\frakM)\subseteq\frakL(\frakN)$, which implies that
$\inf\frakM\leq\inf\frakN$ if both infima exist. In general, there
is little interrelation between the existence of $\inf\frakM$ and the
existence of $\inf\frakN$. This can be seen with the aid of simple
examples constructed on the basis of Kadison's theorem. Let
$\frakM:=\{A,B,C\}$ and $\frakN:=\{A,B\}$. If $A\leq B$ and $A$ and $C$
are not comparable, $\inf\frakN$ exists and $\inf\frakM$ does not. If
$A$ and $B$ are not comparable and $C\leq A$, $C\leq B$, then
$\inf\frakM=C$ exists and $\inf\frakN$ does not.
\end{fact*}

\begin{fact*}[(2.b)]
\label{fa2.b} 
If $\frakM$ and $\frakN$ are nonempty subsets of $\frakS$ such that
$\inf\frakM$ and $\inf\frakN$ exist, then
\begin{equation}
\label{eq2.2}
\frakL(\frakM\cup\frakN) =\frakL\bigl(\{\inf\frakM,\inf\frakN\}\bigr).
\end{equation}
Indeed, $L\in\frakL(\frakM\cup\frakN)$ if and only if $L\leq A$,
$A\in\frakM$, and $L\leq B$, $B\in\frakN$,
which
in turn is equivalent
to $L\leq\inf\frakM$ and $L\leq\inf\frakN$ or, what is the same, 
\[
  L\in \frakL\bigl(\{\inf\frakM,\inf\frakN\}\bigr).
\]
From \eqref{eq2.2} it can be concluded that if $\inf\frakM$ and
$\inf\frakN$ exist, the infimum $\inf(\frakM\cup\frakN)$ exists if
and only if $\inf\frakM$ and $\inf\frakN$ are comparable. If only one
of the infima $\inf\frakM$ and $\inf\frakN$ exists, there are
examples, where $\inf(\frakM\cup\frakN)$ exists as well as examples,
where $\inf(\frakM\cup\frakN)$ does not exist, \cf \nameref{fa2.a}. 
It is an open question, whether $\inf(\frakM\cup\frakN)$ can exist if
both $\inf\frakM$ and $\inf\frakN$ do not exist.
\end{fact*}

\begin{fact*}[(2.c)] 
\label{fa2.c}
For all $S\in\frakS$,
\begin{equation}
  \label{eq2.3}
  \frakL(S+\frakM)=S+\frakL(\frakM).
\end{equation}  
For all bounded and boundedly invertible linear operators $T$ on
$\mathscr H$, 
\begin{equation}
  \label{eq2.4}
  \frakL(T\frakM T^{\ast}) = T \frakL(\frakM) T^{\ast},
\end{equation}
where $T^{\ast}$ denotes the adjoint of $T$.
\end{fact*}

\begin{fact*}[(2.d)]
\label{fa2.d}  
Recall that the strong-operator and weak-operator topologies are the
coarsest topologies such that for all $u,v\in\mathscr H$, the
functions  $S\mapsto Su$, $S\in\frakS$, and $S\mapsto (Su,v)$,
$S\in\frakS$, resp., are continuous. Obviously for a weak-operator
dense subset $\frakN$ of $\frakM$, the relation
$\frakL(\frakM)=\frakL(\frakN)$ is true.
\end{fact*}

\begin{fact*}[(2.e)]
\label{fa2.e}  
If $G\in\frakL\cap\frakM$, then $L\leq G\leq A$, $L\in\frakL$,
$A\in\frakM$. It follows that $\inf\frakM$ exists and is equal to
$G$. In particular, the cardinality of $\frakL\cap\frakM$ is at most
one. If $\frakM$ is countable and weak-operator compact, then
conversely, the existence of $\inf\frakM$ yields
$\frakL\cap\frakM=\{\inf\frakM\}$, \cf Corollary~\ref{cor3.4} below.
\end{fact*}

\begin{fact*}[(2.f)]
\label{fa2.f}
It is not hard to see that $\frakL$ is a convex and
weak-operator closed, hence, strong-operator closed set, cf.\
\cite[Thm.~5.1.2]{KadisonRingrose1983}. Similarly, the set $\frakM$
can be replaced by its convex weak-operator closure without altering
its set of lower bounds.
\end{fact*}

\begin{fact*}[(2.g)]
\label{fa2.g}  
If $\inf\frakM$ does not exist, $\Max\frakL$ is an infinite set. 
\end{fact*}

\begin{proof}
(i) Let $M_{1},M_{2},M_{3}\in\Max\frakL$, $M_{1}\neq M_{2}$, and
$M_{3}\geq(1-\alpha)M_{1}+\alpha M_{2}$ for some $\alpha\in(0,1)$.
Since from the definition of a maximal element immediately follows
that $M_{1}$ and $M_{2}$ are not comparable, we obtain that $M_{3}$
differs from $M_{1}$ and from $M_{2}$. \\
(ii) If $\inf\frakM$ does not exist, by Theorem~\ref{thm2.1} the set
$\Max\frakL$ has at least two elements. Let
$M_{1},M_{2}\in\Max\frakL$, $M_{1}\neq M_{2}$. Since $\frakL$ is
convex, again by Theorem~\ref{thm2.1} one can construct a sequence
$(M_{n})_{n\in\mathbb N}$ of maximal lower bounds satisfying
\begin{equation}
\label{eq2.4a}
M_{n+1}\geq \frac{1}{2}\bigl(M_{1}+M_{n}\bigr),\;n = 2,3,\ldots.
\end{equation}
With the aid of (i) let us show by induction that the elements of this
sequence are pairwise different. Note that \eqref{eq2.4a} yields
\[
M_{n+1} \geq
 \Bigl(1-\frac{1}{2^{n+1-k}}\Bigr)\, M_{1}
  + \frac{1}{2^{n+1-k}}\,M_{k},\; k = 2,\ldots,n.
\]
Therefore, if the operators $M_{1},\ldots,M_{n}$ are pairwise
different by the induction assumption, (i) implies that $M_{n+1}$ is
different from $M_{1}$ and from $M_{k}$, $k=2,\ldots,n$.
\end{proof}

For $S\in\frakS$, denote by $S^{\#}$ its Moore-Penrose inverse. Recall
that $S^{\#}$ is a possibly unbounded self-adjoint operator defined on
$\mathscr N(S)\oplus\mathscr R(S)$ as follows. If $u\in\mathscr N(S)$,
then $S^{\#}u=0$. If $u\in\mathscr R(S)$, then $S^{\#}u=v$, where $v$
is the unique vector from $\overline{\mathscr R(S)}$ satisfying
$Sv=u$. We recall an operator version of Albert's positivity criterion
for $2\times 2$ block matrices and refer to \cite{Shmulyan1959} for
more information. Let $\mathscr H_{1}$ be a subspace of $\mathscr H$
different from $\{0\}$ and $\mathscr H$, and let $\mathscr H_{2}$ be
its orthogonal complement. For $S\in\frakS$, let
\begin{equation}
  \label{eq2.5}
  S = \begin{pmatrix} S_{1} & S_{12} \\ S_{12}^{\ast} & S_{2}
      \end{pmatrix}
\end{equation}
be a $2\times 2$ block operator representation of $S$ according to the
orthogonal decomposition $\mathscr H = \mathscr H_{1}\oplus\mathscr H_{2}$.

\begin{theorem}
  \label{thm2.2}
An operator $S\in\frakS$ is positive if and only if the following three
conditions are satisfied:
\begin{enumerate}
\item[\normalfont(i)] $S_{1}\geq0$,  
\item[(ii)] $\mathscr R(S_{12})\subseteq\mathscr R(S_{1}^{\sfrac{1}{2}})$,
\item[(iii)]
$S_{2}-\bigl((S_{1}^{\#})^{\sfrac{1}{2}}S_{12}\bigr)^{\ast}
(S_{1}^{\#})^{\sfrac{1}{2}}S_{12}\geq0$.
\end{enumerate}
Here, $(S_{1}^{\#})^{\sfrac{1}{2}}$ denotes the unique possibly unbounded
positive self-adjoint square root of $S_{1}^{\#}$.
\end{theorem}

The operator
\begin{equation}
  \label{eq2.6}
 S/S_{1} := S_{2}-\bigl((S_{1}^{\#})^{\sfrac{1}{2}}S_{12}\bigr)^{\ast}
(S_{1}^{\#})^{\sfrac{1}{2}}S_{12} 
\end{equation}
is called the generalized Schur complement of $S$ and plays an important
role in many fields of mathematics, \cf \cite{Zhang2005} for a large
number of applications in the matrix case. A close relative to
$S/S_{1}$ is its extension to the whole space $\mathscr H$ by
\[
  \begin{pmatrix} 0 & 0 \\ 0 & S/S_{1} \end{pmatrix}.
\]  
It was introduced by Anderson \cite{Anderson1971,AndersonTrapp1975},
\cf \cite{Shmulyan1959}, and was called \emph{shorted} operator of
$S$. If  $\frakM$ is a subset of $\frakS^{+}$ and $S\in\frakM$ written
as a $2\times2$ operator matrix \eqref{eq2.5} according to a certain
orthogonal decomposition
$\mathscr H = \mathscr H_{1} \oplus \mathscr H_{2}$, we set
\begin{equation}
  \label{eq2.7}
  \frakM/\mathscr H_{1} := \bigl\{ S/S_{1} : S\in\frakM \bigr\}.
\end{equation}

\section{Generalizations of Kadison's theorem}
\label{sec3}

The aim of the present section is to apply Kadison's elegant  proof
\cite{Kadison1951}, \cf \cite[Ex.~2.8.18]{KadisonRingrose1983}, to a
set $\frakM$ of more than two elements.

\begin{lemma}
\label{lem3.1}
Let $\frakM\subseteq\frakS^{+}$. If there exists a dense subset
$\mathscr D$ of the unit sphere of $\mathscr H$ such that
\begin{equation}
\label{eq3.1}
\inf\bigl\{(Au,u) : A\in\frakM\bigr\}=0,\quad u\in\mathscr D,  
\end{equation}
the infimum of $\frakM$ exists and is equal to $0$. In addition, let
$\frakM$ be a bounded set. If there exists $u\in\mathscr H$ such that
\begin{equation}
  \label{eq3.2}
  \alpha:=\inf\bigl\{(Au,u) : A\in\frakM\bigr\} > 0
\end{equation}
the zero operator is not the infimum of $\frakM$.
\end{lemma}

\begin{proof}
The first claim is an obvious consequence of the continuity of the
operators of $\frakS$. Assume that \eqref{eq3.2} is satisfied. Let
$\mathscr H_{1}$ be the space spanned by $u$, $\mathscr H_{2}$ be its
orthogonal complement, and 
\[
A = \begin{pmatrix} a_{1} & A_{12} \\ A_{12}^{\ast} & A_{2} \end{pmatrix}
\]  
a partition of $A\in\frakM$ according to the decomposition $\mathscr
H=\mathscr H_{1}\oplus\mathscr H_{2}$. Choose $l_{1}\in
(0,\alpha)$. If we can find an operator $L\in\frakL$ of the form
\[
L = \begin{pmatrix} l_{1} & 0 \\ 0 & l_{2}I \end{pmatrix},
\] 
this will mean that $0$ is not the infimum of $\frakM$. By
Theorem~\ref{thm2.2} the inequality $L\leq A$ is true if
\begin{equation}
\label{eq3.3}  
A_{2}-l_{2}I \geq (a_{1}-l_{1})^{-1} A_{12}^{\ast} A_{12}.
\end{equation}
If $\frakM$ is a bounded set, \eqref{eq3.3} is satisfied for all
$A\in\frakM$ and a number $l_{2}\in(-\infty,0)$, where $|l_{2}|$ is
large enough.
\end{proof}  

Using a standard continuity argument, it can be concluded that for a
bounded set $\frakM$, relation~\eqref{eq3.1} is satisfied if and only
if
\begin{equation}
\label{eq3.3a}  
  \inf\bigl\{(Au,u) : A\in\frakM \bigr\} = 0
  \quad\text{for all $u\in\mathscr H$.}
\end{equation}
We omit the details and give a simple example showing that for an
unbounded set $\frakM$, such an equivalence is not true in general.

\begin{example}
\label{ex3.2}

Let $\{u_{n}: n\in\mathbb N\}$ be an orthonormal basis of a Hilbert
space and $P_{n}$ denote the projection onto the space spanned by
$u_{n}$. For the unbounded set $\frakM:=\{n^{2}P_{n} : n\in\mathbb
N\}$ one has
\[
  \inf\bigl\{(n^{2}P_{n}u,u) : n\in\mathbb N\bigr\}=0
\]
if $u$ is a finite linear combination of the basis vectors, therefore,
\eqref{eq3.1} is satisfied, and $\inf\frakM=0$. Let
$(a_{n})_{n\in\mathbb N}$  be a sequence of positive numbers with
\[
 \sum_{n=1}^{\infty}a_{n}^{2}<\infty \quad\text{and}\quad
 \limsup_{n\to\infty} na_{n} =\infty. 
\]
Then, for $u:=\sum_{n=1}^{\infty}a_{n}u_{n}$ one has
\[
\limsup_{n\to\infty}(n^{2} P_{n}u,u)=
\limsup_{n\to\infty}n^{2}a_{n}^{2} = \infty,
\]
and \eqref{eq3.3a} is not satisfied.
\end{example}

Our proof of the second assertion of Lemma~\ref{lem3.1} is based on
Kadison's idea. It does not work for an unbounded set $\frakM$ and the
following question remains open. Does there exist an unbounded set
$\frakM$ with $\inf\frakM=0$ although for all dense subsets $\mathscr
D$ of the unit sphere, relation \eqref{eq3.1} fails?

The following theorem, whose first assertion was stated by
{\v S}mul{\cprime}jan \cite[Thm. 1.2]{Shmulyan1959} in a similar form,
is an immediate consequence of Lemma~\ref{lem3.1} and \eqref{eq2.3}.

\begin{theorem}
  \label{thm3.3}
If there exist $G\in\frakS$ and a dense subset $\mathscr D$ of the
unit sphere of $\mathscr H$ such that for $u\in\mathscr D$,
\begin{equation}
  \label{eq3.4}
  (Gu,u)=\inf\bigl\{(Au,u) : A\in\frakM\bigr\},
\end{equation}
the infimum of $\frakM$ exists and equals $G$. Let, additionally,
$\frakM$ be bounded. If $\inf\frakM$ exists, condition~\eqref{eq3.4} is
satisfied for $G:=\inf\frakM$ and all $u\in\mathscr H$.
\end{theorem}

As an application of the preceding theorem we can generalize
Kadison's result.

\begin{corollary}
\label{cor3.4}
Let $\frakM$ be countable and weak-operator compact. The infimum of
$\frakM$ exists if and only if there exists $G\in\frakM$ such that
$G\leq A$ for all $A\in\frakM$. In this case, $G=\inf\frakM$.
\end{corollary}

\begin{proof}
The `if'-part is obvious and is true for an arbitrary set $\frakM$.
Conversely, assume that $G:=\inf\frakM$ exists. Choose $v\in\mathscr
H$. Since $\frakM$ is weak-operator compact, it is weak-operator
bounded, hence, bounded by uniform boundedness principle. Boundedness
of $\frakM$ and Theorem~\ref{thm3.3} yield the existence of                             
$A_{n}\in\frakM$ with
\[
(Gv,v)\leq (A_{n}v,v)  \leq (Gv,v)+\frac{1}{n},\;n\in\mathbb N.
\]  
By weak-operator compactness one can find a subsequence
$\bigl(A_{n_{k}}\bigr)$ and an operator $A\in\frakM$ satisfying
\[
(Gv,v) = \lim_{k\to\infty}(A_{n_{k}}v,v) = (Av,v).
\]  
Since $A-G\in\frakS^{+}$, it follows $v\in\mathscr N(A-G)$, and since
$v\in\mathscr H$ was arbitrary, we obtain
\begin{equation}
\label{eq3.5}
\mathscr H = \bigcup_{A\in\frakM}\mathscr N(A-G).
\end{equation}
If $A-G\neq 0$, the set $\mathscr N(A-G)$ is a nowhere dense subset of
$\mathscr H$. Therefore, \eqref{eq3.5} and Baire's category theorem
imply that there exists $A\in\frakM$ with $G=A$. The equality
$G=\inf\frakM$ is trivial.
\end{proof}

\begin{remark}
If $\frakM$ is a finite set, the topological proof of the preceding
corollary on the basis of Baire's category theorem can be replaced by
a purely algebraic proof of the fact, that no vector space over the
field $\mathbb R$ or $\mathbb C$ is a finite union of proper subspaces.
\end{remark}

We show by examples that the `only if'-part of corollary~\ref{cor3.4}
is not true if one of the conditions on $\frakM$ is omitted.

\begin{example}
\label{ex3.5}
Each of the following three sets $\frakM$ satisfies $\inf\frakM=0$ and
$0\notin\frakM$.
\begin{enumerate}
\item[(i)]
Let $\{u_{n} : n\in\mathbb N\}$ be a countable dense subset of the
unit  sphere of $\mathbb C^{2}$ and $P_{n}$ denote the projection onto
the space spanned by $u_{n}$. The set $\frakM:=\{P_{n}:n\in\mathbb
N\}$ is countable and bounded but not closed.
\item[(ii)]
Let $\frakM=\{P_{u} : u\in\mathbb C^{2}, \|u\|=1\}$, where $P_{u}$
denote the projection onto the space spanned by $u$. The set
$\frakM$ is compact but not countable.
\item[(iii)]
The set $\frakM$ of Example~\ref{ex3.2} is countable and
weak-operator closed but not bounded.
\end{enumerate}
\end{example}

\section{Maximal lower bounds}
\label{sec4}

The present section deals with the structure of the set of maximal
lower bounds of a set $\frakM$. We start with an auxiliary result,
which is also of interest in the theory of quantum effects, cf.\
Section~\ref{sec6}.
\begin{lemma}
\label{lem4.1}
Let $\frakM\subseteq\frakS^{+}$. If
\begin{equation}
\label{eq4.1}
\bigcap_{A\in\frakM} {\mathscr R}(A^{\sfrac{1}{2}}) = \{0\}  
\end{equation}
the zero operator is the only positive lower bound of $\frakM$. Let,
additionally, $\frakM$ be a finite set. If \eqref{eq4.1} is not true,
there exists a positive lower bound of $\frakM$ different from $0$.
\end{lemma}

\begin{proof}
If $S,T\in\frakS^{+}$ and $S\leq T$, then
$\mathscr R(S^{\sfrac{1}{2}})\subseteq\mathscr R(T^{\sfrac{1}{2}})$ by
Douglas' theorem \cite{Douglas1966}. Thus, \eqref{eq4.1} yields $L=0$ if
$L\in\frakL\cap\frakS^{+}$. The proof of the second assertion uses
properties of the parallel sum of positive operators, cf.\
\cite{AndersonTrapp1975,FillmoreWilliams1971,PekarevSmuljan1976} for
definition  and basic properties of parallel sum. Let $\frakM$ be
finite and $S$ denote the parallel sum of  the operators of $\frakM$.
Since $S$ is a positive lower bound of $\frakM$ satisfying
\[
  \mathscr R(S^{\sfrac{1}{2}}) =
  \bigcap_{A\in\frakM}\mathscr R(A^{\sfrac{1}{2}}), 
\]
\cf \cite[Thm.~4.2]{FillmoreWilliams1971} or
\cite[Thms.~10, 11]{AndersonTrapp1975}, the second assertion follows.
\end{proof}

Since there exist $A,B\in\frakS^{+}$ with
$\mathscr N(A)=\mathscr N(B) =\{0\}$ and
$\mathscr R(A^{\sfrac{1}{2}})\cap \mathscr R(B^{\sfrac{1}{2}})=\{0\}$, 
the second assertion of the preceding lemma shows that Theorem~2 of
\cite{Aslanov2005} is not correct.

Note that $M\in\frakL$ is a maximal lower bound of $\frakM$ if and
only if the zero operator is the only positive lower bound of the set
$\frakM-M$. Therefore, from Lemma~\ref{lem4.1} a characterization of
$\Max\frakL$ can be derived immediately.

\begin{theorem}
\label{thm4.2}
Let $\frakM\subseteq\frakS$. If $M\in\frakL$ and
\begin{equation}
\label{eq4.2}
\bigcap_{A\in\frakM} \mathscr R\bigl((A-M)^{\sfrac{1}{2}}\bigr)=\{0\},
\end{equation}
the operator $M$ is a maximal lower bound of $\frakM$. Let,
additionally, $\frakM$ be a finite set. If \eqref{eq4.2} is not true,
the operator $M$ is not a maximal element of $\frakL$.
\hfill\MYQED
\end{theorem}

\begin{corollary}  
\label{cor4.3a}
Let $\mathscr H=\mathscr H_{1}\oplus\mathscr H_{2}$ be a nontrivial
orthogonal decomposition of $\mathscr H$. Let $\frakM_{1}$ be a subset
of $\frakS(\mathscr H_{1})$ bounded from below and
\[
  \frakM:= \left\{
    \begin{pmatrix}
      A_{1} & 0 \\ 0 & 0
    \end{pmatrix} : A_{1} \in\frakM_{1}\right\}
\]  
a corresponding subset of $\frakS(\mathscr H)$. If $M_{1}$ is a
maximal lower bound of $\frakM_{1}$, then
$\begin{psmallmatrix} M_{1} & 0 \\ 0 & 0 \end{psmallmatrix}$
is a maximal lower bound of $\frakM$. Let, additionally,
$\frakM_{1}$ be a finite set. Any maximal lower bound $M$ of $\frakM$
has the form
$\begin{psmallmatrix} M_{1} & 0 \\ 0 & 0 \end{psmallmatrix}$
for some maximal lower bound $M_{1}$ of $\frakM_{1}$.
\end{corollary}

\begin{proof}
The first assertion is obvious. Let $\frakM_{1}$ be finite and
\[
M = \begin{pmatrix} M_{1} & M_{12} \\ M_{12}^{\ast} & M_{2} \end{pmatrix}
\]  
be a maximal lower bound of $\frakM$. Since
$\mathscr R(S)\subseteq\mathscr R(S^{\sfrac{1}{2}})$ for
$S\in\frakS^{+}(\mathscr H)$, from Theorem~\ref{thm4.2} follows
\[
  \bigcap_{A_{1}\in\frakM_{1}}
  \mathscr R\left(\begin{pmatrix} A_{1} & 0 \\ 0 & 0 \end{pmatrix}- M \right)
    = \{0\}.
\]  
This condition would not be satisfied if
$\begin{psmallmatrix} M_{12} \\ M_{2}\end{psmallmatrix}\neq 0$. Thus
$M=\begin{psmallmatrix} M_{1} & 0 \\ 0 & 0 \end{psmallmatrix}$
and, clearly, $M_{1}$ is a maximal lower bound of $\frakM_{1}$.
\end{proof}

\begin{corollary}
\label{cor4.3}
Let $\frakM\subseteq\frakS$. If $M\in\frakL$ and
\begin{equation}
\label{eq4.3}
\overline{\sum_{A\in\frakM}\mathscr N(A-M)} = \mathscr H,  
\end{equation}
the operator $M$ is a maximal element of $\frakL$. If, additionally,
$\frakM$ is a finite set and the ranges of all operators $A-M$,
$A\in\frakM$, are closed, then \eqref{eq4.3} is also necessary for the
maximality of $M$. 
\end{corollary}
\begin{proof}
For a subset $\mathscr F$ of $\mathscr H$ set $\mathscr F^{\perp}:=
\{v\in\mathscr H:(u,v)=0\text{ for all }u\in\mathscr F\}$. Recall that
an arbitrary family $\{\mathscr F_{\lambda}:\lambda\in\Lambda\}$ of
subsets of $\mathscr H$ satisfies the relation
\[
  \Bigl(\sum_{\lambda\in\Lambda} \mathscr F_{\lambda}\Bigr)^{\perp}
  = \bigcap_{\lambda\in\Lambda} \mathscr F_{\lambda}^{\perp}.
\]  
Writing this equality down for the family
$\{\mathscr N(A-\frakM):A\in\frakM\}$, we obtain  
\begin{equation}
\label{eq4.3a}
\Bigr(\sum_{A\in\frakM}\mathscr N(A-M)\Bigr)^{\perp} =
\bigcap_{A\in\frakM}\bigl(\mathscr N(A-M)\bigr)^{\perp}.
\end{equation}
Noting that
\[
\Bigl(\mathscr N(A-M)\Bigr)^{\perp} =
\Bigl(\mathscr N\bigl((A-M)^{\sfrac{1}{2}}\bigr)\Bigr)^{\perp} =
\overline{\mathscr R\bigl((A-M)^{\sfrac{1}{2}}\bigr)},
\]
and taking the orthogonal complements on both sides of \eqref{eq4.3a},
we obtain that relation \eqref{eq4.3} and the equality
\begin{equation}
\label{eq4.3b}
\bigcap_{A\in\frakM}
\overline{\mathscr R\bigl((A-M)^{\sfrac{1}{2}}\bigr)}
= \{0\}
\end{equation}
are equivalent. It follows that \eqref{eq4.3} yields \eqref{eq4.2} and
the first assertion of the corollary is a consequence of the first
assertion of Theorem~\ref{thm4.2}. If for all $A\in\frakM$ the range
of $A-M$ is closed, we have
\[
\mathscr R\bigl((A-M)^{\sfrac{1}{2}}\bigr) =
\overline{\mathscr R\bigl((A-M)^{\sfrac{1}{2}}\bigr)}
\]
since 
\[
\mathscr R(A-M) \subseteq \mathscr R\bigl((A-M)^{\sfrac{1}{2}}\bigr)
\subseteq\overline{\mathscr R(A-M)}
\quad\text{and}\quad
\mathscr R(A-M) = \overline{\mathscr R(A-M)}.
\]  
Thus, \eqref{eq4.2} and \eqref{eq4.3} are equivalent in this case, and
the second assertion of the corollary follows from the second
assertion of Theorem~\ref{thm4.2}
\end{proof}

\begin{remark}
\label{rem4.4}  
In general \eqref{eq4.3} does not follow from \eqref{eq4.2} if the
ranges of $A-M$, $A\in\frakM$, are not closed, cf. remarks following
Lemma~\ref{lem4.1}. It is also interesting to compare the condition
\eqref{eq4.3} with the characterization \eqref{eq3.5} of the greatest
element.  
\end{remark}

We give an example, which indicates that there is only little chance to
weaken the finiteness condition of the second assertion of
Theorem~\ref{thm4.2}.

\begin{example}
\label{ex4.3}
The set 
\[
  \frakM := \left\{
    \begin{mypmatrix}
      1+\frac{1}{n} & \frac{1}{\sqrt{n}} \\
      \frac{1}{\sqrt{n}} & \frac{1}{n}
    \end{mypmatrix}
    : n\in\mathbb N\right\} \cup
  \left\{
    \begin{pmatrix}
      1 & 0 \\ 0 & 0
    \end{pmatrix}
  \right\}
\]
is a countable and compact subset of the set of $2\times2$ matrices.
If $L\in\frakL\cap\frakS^{+}$, it has the form
\[
  L =\begin{pmatrix} l & 0 \\ 0 & 0 \end{pmatrix}
  \quad\text{for some $l\in[0,1]$}.
\]
The determinant
\[
\det\left(
  \begin{mypmatrix}
    1+\frac{1}{n} & \frac{1}{\sqrt{n}} \\[3pt]
    \frac{1}{\sqrt{n}} & \frac{1}{n}
  \end{mypmatrix} - L\right) = \frac{1}{n^{2}}-\frac{l}{n}
\]
is negative for $n$ large enough if $l>0$. Thus, $0$ is the only
positive lower bound of $\frakM$ and $M=0$ is a maximal lower bound of
$\frakM$. If $M=0$, the left-hand side of \eqref{eq4.3} is
$\left\{\begin{psmallmatrix}0\\z\end{psmallmatrix}:z\in\mathbb C\right\}$,
which is different from $\mathbb C^{2}$.
\end{example}

The preceding example shows that the formulation of
Theorem~1.15 of \cite{Stott2017dis} is not quit correct, see, however,
its corrected version \cite[Thm.~1]{GaubertStott2017}.

Recall that, by definition, $L\in\frakL$ is an \emph{extreme point} of
the convex set $\frakL$ if equality $L=\frac{1}{2}(L_{1}+L_{2})$ for
some $L_{1},L_{2}\in\frakL$ implies that $L_{1}=L_{2}$. As an
application of Theorem~\ref{thm4.2} and relation~\eqref{eq4.3} we
establish a relation between the set of extreme points and the set of
maximal elements of $\frakL$.

\begin{theorem}
\label{thm4.4}
For a set $\frakM$, the set of extreme points of $\frakL$ is a subset
of the set $\Max\frakL$. If $M$ is a maximal element of $\frakL$
satisfying \eqref{eq4.3}, it is an extreme point of $\frakL$.
\end{theorem}

\begin{proof}
If $L\in\frakL$ is not maximal and $M$ is a maximal element with
$L\leq M$, one can write
\[
  L=\tfrac{1}{2}\bigl[(2L-M)+M\bigr],
\]
where
$2L-M=L+(L-M)$ and $M$ are different elements of $\frakL$. Thus $L$
can't be an extreme point.
 
To prove the second assertion, note first that if $M$ is maximal and
$M=\frac{1}{2}(M_{1}+M_{2})$ with $M_{1},M_{2}\in\frakL$, then
necessarily $M_{1}$, $M_{2}$ are maximal elements. Recall that
$\mathscr N(S+T)=\mathscr N(S)\cap\mathscr N(T)$ for
$S,T\in\frakS^{+}$. If $M$ satisfies \eqref{eq4.3}, it follows
\begin{align*}
    \overline{\sum_{A\in\frakM} \mathscr N(A-M)} 
& = \overline{\sum_{A\in\frakM} \mathscr N
    \bigl(A-\tfrac{1}{2}(M_{1}+M_{2})\bigr)} \\
& = \overline{\sum_{A\in\frakM} \bigl(\mathscr N(A-M_{1})
    \cap\mathscr N(A-M_{2)})\bigr)}
  = \mathscr H.
\end{align*}
Since $M_{1}u=M_{2}u$, if $u\in\mathscr N(A-M_{1})\cap\mathscr
N(A-M_{2})$, we obtain $M_{1}=M_{2}$.
\end{proof}

The following discussion is motivated by
\cite[Sect.~4]{Stott2016arxiv}. Let $u$ be a unit vector of $\mathscr
H$. Set
\[
  \alpha :=\inf\bigl\{(Au,u):A\in\frakM\bigr\}, 
\]
and
\[  
  \frakM_{u} :=\bigl\{A\in\frakM:(Au,u)=\alpha\bigr\}, \quad
  \frakL_{u} :=\bigl\{L\in\frakL:(Lu,u)=\alpha\bigr\}.
\]
Let $\mathscr H_{1}$ be the space spanned by $u$ and $\mathscr H_{2}$
be its orthogonal complement. Denote by
\[
  S= \begin{pmatrix}
  s_{1} & S_{12} \\ S_{12}^{\ast} & S_{2}
  \end{pmatrix}
\]
a block operator representation of $S\in\frakS$ according to the
orthogonal decomposition $\mathscr H=\mathscr H_{1}\oplus\mathscr H_{2}$.

\begin{theorem}
\label{thm4.5}
Let $\frakM_{u}$ be a nonempty set. If $\frakL_{u}$ is not empty, then
\begin{equation}
\label{eq4.4}
Au = Bu,\;A,B\in\frakM_{u}.
\end{equation}
\end{theorem}

\begin{proof}
For
\[
A=\begin{pmatrix}\alpha & A_{12}\\A_{12}^{\ast} & A_{2}\end{pmatrix}
\in\frakM_{u}\quad\text{and}\quad
L=\begin{pmatrix}\alpha & L_{12}\\L_{12}^{\ast} & L_{2}\end{pmatrix}
\in\frakL_{u},
\]
inequality $L\leq A$ yields $A_{12}^{\ast}=L_{12}^{\ast}$ by
Theorem~\ref{thm2.2}~(ii). It follows $Au=Lu$ for all $A\in\frakM_{u}$.
\end{proof}

Let $\frakM_{u}$ be not empty and \eqref{eq4.4} be satisfied. For
\[
B=\begin{pmatrix}\alpha & B_{12}\\B_{12}^{\ast} & B_{2}\end{pmatrix}
\in\frakM_{u}
\]
define
\begin{equation}
\label{eq4.5}
\frakM' := \{A_{2}-(a_{1}-\alpha)^{\#}(A_{12}-B_{12})^{\ast}(A_{12}-B_{12})
           : A\in\frakM\}.
\end{equation}
By \eqref{eq4.4} the definition of $\frakM'$ does not depend on the
choice of $B$. From Theorems~\ref{thm4.4} and \ref{thm2.2} one can
immediately derive the following description of the set $\frakL_{u}$,
which generalizes a result of \cite[Prop.~4.6]{Stott2016arxiv}.

\begin{theorem}
\label{thm4.6}
Let $\frakM_{u}$ be not empty and \eqref{eq4.4} be satisfied. An
operator $L$ is an element of $\frakL_{u}$ if and only if its block
operator representation is of the form
\begin{equation}
\label{eq4.6}
L=\begin{pmatrix}\alpha & B_{12}\\B_{12}^{\ast} & L_{2}\end{pmatrix}
\end{equation}
for some $L_{2}\in\frakL(\frakM')$. In particular, if $\frakM'$ is not
bounded from below, the set $\frakL_{u}$ is empty.
\end{theorem}

As the following example reveals, it is possible that
$\frakL(\frakM')$ is empty even if \eqref{eq4.4} is satisfied. 

\begin{example}
\label{ex4.7}
The set 
\[
  \frakM := \left\{
    \begin{mypmatrix}
      1+\frac{1}{n^{2}} & \frac{1}{\sqrt{n}} \\
      \frac{1}{\sqrt{n}} & \frac{1}{n}
    \end{mypmatrix}
    : n\in\mathbb N\right\} \cup
  \left\{
    \begin{pmatrix}
      1 & 0 \\ 0 & 0
    \end{pmatrix}
  \right\}
\]
is a countable and compact subset of the set of $2\times 2$
matrices. Choosing $u:=\begin{psmallmatrix} 1 \\
  0 \end{psmallmatrix}$, we have $\alpha=1$,
\[
  \frakM_{u}=\left\{ \begin{pmatrix}
     1 & 0 \\ 0 & 0
    \end{pmatrix}\right\},
\]
and \eqref{eq4.4} is satisfied. The set
\[
  \frakM'=\left\{\frac{1}{n}-\left(\frac{1}{n^{2}}\right)^{2}
  \frac{1}{\sqrt{n}}\frac{1}{\sqrt{n}}
  : n\in\mathbb N\right\}\cup\{0\} =
  \left\{\frac{1}{n}-n : n\in\mathbb N\right\} \cup\{0\}
\]  
is not bounded from below and according to Theorem~\ref{thm4.5}, the
set $\frakL_{u}$ must be empty. This can also be seen directly. If
\[
L =   
\begin{pmatrix}
     l_{1} & l_{12} \\ l_{12}^{\ast} & l_{2}
    \end{pmatrix},
\]
would be a lower bound of $\frakM$ satisfying $(Lu,u)=1$, we must have
$l_{1}=1$ and the determinant
\[
\det \left(
  \begin{mypmatrix}
  1+\frac{1}{n^{2}} & \frac{1}{\sqrt{n}} \\
  \frac{1}{\sqrt{n}} & \frac{1}{n}
  \end{mypmatrix}
  - L \right) =
\frac{1}{n^{2}}- \frac{l_{2}}{n^{2}} - \left|\frac{1}{\sqrt{n}}-l_{12}\right|^{2}
\]  
must be nonnegative for all $n\in\mathbb N$, a contradiction.
\end{example}

If $\frakM_{u}$ is empty, condition \eqref{eq4.4} is trivially
satisfied. Also in this case, there are examples, where $\frakL_{u}$
is not empty, as well as examples, where $\frakL_{u}$ is empty.

\begin{example}
\label{ex4.8}
(i) If
\[
  \frakM := \left\{
    \begin{pmatrix}
      1+\frac{1}{n} & 1 \\
      1             & 1
    \end{pmatrix}
    : n\in\mathbb N\right\} \quad\text{and}\quad
  u:= \begin{pmatrix} 1 \\ 0 \end{pmatrix},
\]
the set $\frakM_{u}$ is empty. The set $\frakL_{u}$ is not empty since
\[
  \begin{pmatrix} 1 & 1 \\ 1 & 1 \end{pmatrix} \in\frakL_{u}.
\]
(ii) The set
\[
  \frakM := \left\{
    \begin{pmatrix}
      1+\frac{1}{n} & 1 \\
      1             & 1
    \end{pmatrix}
    : n\in\mathbb N\right\}
    \cup
    \left\{
    \begin{pmatrix}
      1+\frac{1}{n} & 2 \\
      2             & 4
    \end{pmatrix}
    : n\in\mathbb N\right\}
\]  
and its closure
\[
  \overline{\frakM} =
  \frakM \cup
  \left\{
     \begin{pmatrix} 1 & 1 \\ 1 & 1 \end{pmatrix}, 
     \begin{pmatrix} 1 & 2 \\ 2 & 4 \end{pmatrix}
  \right\}
\]  
have the same lower bounds. If $u:= \begin{psmallmatrix}1\\0\end{psmallmatrix}$
one obtains
\[
  \frakM_{u} =
  \left\{
    \begin{pmatrix} 1 & 1 \\ 1 & 1 \end{pmatrix},
    \begin{pmatrix} 1 & 2 \\ 2 & 4 \end{pmatrix}
  \right\}
\]  
and $\frakL_{u}(\frakM)=\frakL_{u}(\overline\frakM)$ is empty by
Theorem~\ref{thm4.5}. 
\end{example}

The following description of the maximal elements of $\frakL_{u}$ is
an immediate consequence of Theorem~\ref{thm4.5}.

\begin{corollary}
\label{cor4.9}
Let $M$ be a lower bound of $\frakM$. Under the conditions of
Theorem~\ref{thm4.5} the following assertions are equivalent:
\begin{enumerate}
\item[(i)]
the operator $M$ is a maximal element of $\frakL_{u}$,
\item[(ii)]  
the operator $M$ is a maximal element of $\frakL$ and satisfies
$(Mu,u)=\alpha$,
\item[(iii)]
the operator $M$ has a block operator representation
\[
M =   
\begin{pmatrix} \alpha & B_{12} \\ B_{12}^{\ast} & M_{2}\end{pmatrix},
\text{ where }
\begin{pmatrix} \alpha & B_{12} \\ B_{12}^{\ast} & B_{2}\end{pmatrix}
\in\frakM_{u}
\]  
and $M_{2}$ is a maximal lower bound of $\frakM'$.
\end{enumerate}
\end{corollary}

\begin{remark}
\label{rem4.10}
Let $\mathscr H$ be separable and $\frakM$ be a bounded subset of
$\frakS$. The operator norm closure $\overline{\frakM}$ of $\frakM$ is
weak-operator compact, cf.\
\cite[Thm.~5.1.3]{KadisonRingrose1983}. Analyzing the proof of
Corollary~\ref{cor3.4} above, we conclude that for the weak-operator
compact set $\overline{\frakM}$ and an arbitrary unit vector $u$, the
set $(\overline{\frakM})_{u}$ is not empty. Since
$\frakL(\frakM)=\frakL(\overline{\frakM})$ by \nameref{fa2.d}, the set $\frakM$
can be replaced by $\overline{\frakM}$, which shows that for a
separable Hilbert space $\mathscr H$, Theorems~\ref{thm4.5}, 
\ref{thm4.6} and Corollary~\ref{cor4.9} can be applied to an arbitrary
bounded subset $\frakM$ of $\frakS$ with obvious modifications.
\end{remark}

We conclude the present section by proving the existence of a positive
maximal lower bound for an arbitrary set of positive $n\times n$
matrices without referring to Zorn's lemma explicitly. The proof is
by induction over the order $n$ of the matrices and could serve as a
basis of an algorithm for computing a positive maximal lower bound.

\begin{theorem}
\label{thm4.11}
Let $n\in\mathbb N$ and $\frakM$ be a set of positive $n\times n$
matrices. There exists a positive maximal lower bound of $\frakM$.
\end{theorem}

\begin{proof}
Induction over $n$. If $n=1$ a positive maximal lower bound exists and
is even an infimum. Let $\frakM$ be a set of positive
$(n+1)\times(n+1)$ matrices and let
\[
\gamma:=\inf\bigl\{(Au,u):A\in\frakM,u\in\mathscr H,\|u\|=1\bigr\}.
\]
Replacing $\frakM$ by $\frakM-\gamma I$, we can assume that $\gamma=0$.
Let $(A_{k})$ be a sequence of matrices of $\frakM$ and $(u_{k})$ be a
sequence of unit vectors of $\mathbb C^{n+1}$ satisfying
\[
  (A_{k}u_{k},u_{k)}) < \frac{1}{k}.
\]
Choosing a subsequence if necessary, we can also assume that a unit
vector  $u:=\lim_{k\to\infty}u_{k}$ exists. If $L$ is a positive lower
bound of  $\frakM$, it follows
\[
  0\leq(Lu,u)=\lim_{k\to\infty}(Lu_{k},u_{k})
  \leq\lim_{k\to\infty}(A_{k}u_{k},u_{k})=0,
\]
hence, $u\in\mathscr N(L)$. Let $\mathscr H_{1}$ be the space spanned
by $u$ and $\mathscr H_{2}$ its orthogonal complement. Write
\[
  A=\begin{pmatrix}a_{1}&A_{12}\\A_{12}^{\ast}&A_{2}\end{pmatrix}
\]
according to the orthogonal decomposition $\mathbb C^{n+1}=\mathscr
H_{1}\oplus\mathscr H_{2}$. From Theorem~\ref{thm2.2} (iii) it follows
that the set
\[
\frakM/\mathscr H_{1}:=\bigl\{
A/a_{1}=A_{2}-a_{1}^{\#}A_{12}^{\ast}A_{12} : A\in\frakM\bigr\} 
\]
is a set of positive $n\times n$ matrices. By the induction assumption,
there exists a positive maximal lower bound of $\frakM/\mathscr
H_{1}$. Denote it by $M_{2}$. It follows that $A/a_{1}-M_{2}\geq0$,
$A\in\frakM$, and again by Theorem~\ref{thm2.2} that
\[
M:=\begin{pmatrix}0&0\\0&M_{2}\end{pmatrix}
\]
is a positive lower bound of $\frakM$. To show that $M$ is maximal, let
\[
L:=\begin{pmatrix}l_{1}&L_{12}\\L_{12}^{\ast}&L_{2}\end{pmatrix}
\]
be a lower bound of $\frakM$ with $M\leq L$. Since $L$ is positive, it
has the form
\[
L:=\begin{pmatrix}0&0\\0&L_{2}\end{pmatrix}
\]
where $A/a_{1}-L_{2}\geq0$, $A\in\frakM$. Thus, $L_{2}$ is a lower
bound of $\frakM/\mathscr H_{1}$ and $M_{2}\leq L_{2}$. Since $M_{2}$
is a maximal lower bound, it follows that $L_{2}=M_{2}$, which yields
$M=L$.  
\end{proof}  

\section{Maximal lower bounds of a two-element set}
\label{sec5}
Let $\mathscr H$ be an arbitrary Hilbert space over $\mathbb C$. If
$\frakM:=\{A,B\}$ is a set of two operators on $\mathscr H$, some
maximal lower bounds can be expressed in terms of $A$ and $B$
explicitly. For a bounded operator $T$ on $\mathscr H$ set $|T|:=(T^{\ast}T)^{\sfrac{1}{2}}$.

\begin{theorem}
\label{thm5.1}  
For any bounded and boundedly invertible linear operator $T$ on
$\mathscr H$, the operator
\begin{equation}
\label{eq5.1}
M_{T}:=\frac{1}{2}\Bigl(A+B
-T^{\ast}\bigl|(T^{\ast})^{-1}(A-B)T^{-1}\bigr| T\Bigr) 
\end{equation}
is a maximal lower bound of the set $\frakM=\{A,B\}$.
\end{theorem}

\begin{proof}
Let $T=I$ and
\[
  A-B = \int_{a}^{b} t\,E(dt)
\]
be the spectral representation of the self-adjoint operator $A-B$. It
is easy to see that
\[
  \mathscr N(A-M_{I})=\mathscr R\bigl(E([a,0])\bigr)
  \quad\text{and}\quad
  \mathscr N(B-M_{I})=\mathscr R\bigl(E([0,b])\bigr),
\]
which yields $M_{I}\in\Max\frakL$ by \eqref{eq4.3}. The general result
is a consequence of \eqref{eq2.4}.
\end{proof}

\begin{corollary}
\label{cor5.2}
For a set $\frakM=\{A,B\}$ of two operators, there exists a maximal
lower bound $M$ satisfying
$\mathscr N(A-M)+\mathscr N(B-M)=\mathscr H$.
\end{corollary}

Note that for a unitary operator $U$ on $\mathscr H$ the equality
\[
  U^{\ast}\bigl|U(A-B)U^{\ast}\bigr|U = |A-B|
\]
is true. If $T$ is bounded and boundedly invertible, it has a polar
decomposition $T=U|T|$, where $U$ is unitary. It follows that
$M_{T}=M_{|T|}$, and the sets
\[
\bigl\{M_{T}: \text{T is bounded and boundedly invertible}\}
\]
and
\begin{equation}
\label{eq5.2}
\bigl\{ M_{S} : \text{$S\in\frakS^{+}$, $S$ is boundedly invertible}\bigr\} 
\end{equation}
are equal.

Stott \cite[Thm.~2.3]{Stott2017dis}, cf.\
\cite[Thm.~1]{GaubertStott2017}, has shown by solving a linear maximum
problem that the set described by \eqref{eq5.2} coincides with the set
$\Max\frakL$ if $\dim\mathscr H<\infty$. Note, however, that in
general the correspondence $S\mapsto M_{S}$ is not one-to-one. For if
$S$ leaves invariant the spaces $\mathscr R\bigl(E[a,0)\bigr)$,
$\mathscr R\bigl(E(\{0\})\bigr)$, and $\mathscr R\bigl(E(0,b]\bigr)$,
the operator $M_{S}$ is equal to $M_{I}$. If $\dim\mathscr H=\infty$,
the set \eqref{eq5.2} is, in general, a proper subset of
$\Max\frakL$. To see this, choose $A,B\in\frakS^{+}$ such that
$\mathscr N(A)=\mathscr N(B)=\{0\}$ and
$\mathscr R(A^{\sfrac{1}{2}})\cap\mathscr R(B^{\sfrac{1}{2}})=\{0\}$.
By Lemma~\ref{lem4.1} the operator $0$ is the only positive lower
bound of the set $\frakM=\{A,B\}$, hence $0\in\Max\frakL$
and $\mathscr N(A-0)+\mathscr N(B-0)=\{0\}$. Since for every $M_{T}$
of \eqref{eq5.1}, the relation
\[
  \overline{\mathscr N(A-M_{T})+\mathscr N(B-M_{T})}=\mathscr H
\]  
is satisfied, the operator $0$ differs from all
operators $M_{T}$ of the form \eqref{eq5.1}. Note that in the just
mentioned example all operators $M_{T}$ are not positive. To explain
this, recall that the function $T\mapsto|T|$ satisfies the triangle
inequality only in a weakened form,
\cf \cite[Thm. III.5.6]{Bhatia1997}.

Setting $T=I$, formula \eqref{eq5.1} generalizes the well known
formula  for the minimum of two real numbers, and the following
continuity assertion can be derived immediately. The statement of a
similar  measurability result is omitted.

\begin{corollary}
\label{cor5.3}
Let $(\mathscr T,\tau)$ be a topological space with topology $\tau$ and
$\tau_{1}$ be a vector topology on $\frakS$ such that the map
$S\mapsto |S|$, $S\in\frakS$, is continuous. For any
$(\tau,\tau_{1})$-continuous functions $\Phi:\mathscr T\to\frakS$ and
$\Psi:\mathscr T\to\frakS$, there exists a $(\tau,\tau_{1})$-continuous
function $M:\mathscr T\to\frakS$ with
\[
 M(t)\in\Max\frakL\bigl(\{\Phi(t),\Psi(t)\}\bigr),\;t\in\mathscr T. 
\]
\end{corollary}

\begin{problem*}
Extend Corollaries~\ref{cor5.2} and \ref{cor5.3} from a set of two
elements to an arbitrary finite set.
\end{problem*}

Taking into account \eqref{eq2.3}, we can suppose that a set $\frakM$
of two operators has the form $\frakM=\{A,0\}$ for some $A\in\frakS$.
Moreover, by Corollary~\ref{cor4.3a} it can be presumed without loss of
generality that $\mathscr R(A)$ is a dense subset of $\mathscr H$ or,
what is the same, that
\begin{equation}
  \label{eq5.3}
  \mathscr N(A)=\{0\}.
\end{equation}

The following auxiliary result was used by Stott in the matrix case to
prove Theorem~4.3 of \cite{Stott2016arxiv}.

\begin{lemma}
\label{lem5.4}
Let $\mathscr K$ be a subspace of $\mathscr H$.
\begin{enumerate}
\item[(i)]
  If there exists $M\in\Max\frakL$ satisfying $Mu=Au$, $u\in\mathscr
  K$, then $(Au,u)<0$ for all $u\in\mathscr K\setminus\{0\}$.
\item[(ii)]
  If there exists $M\in\Max\frakL$ such that $\mathscr
  K\subseteq\mathscr N(M)$, then $(Au,u)>0$ for all  $u\in\mathscr
  K\setminus\{0\}$.
\end{enumerate}
\end{lemma}

\begin{proof}
(i) Since $M\leq0$, we have $(Au,u)\leq0$, $u\in\mathscr K$. If
$u\in\mathscr K$ and $(Au,u)=0$, then it follows $(Mu,u)=0$, hence
$u\in\mathscr N(-M)=\mathscr N(M)$ since $-M\in\frakS^{+}$. Thus
$Au=Mu=0$ and $u=0$ by \eqref{eq5.3}. \\
(ii) is proved in a similar way by changing the roles of $A$ and $B$.  
\end{proof}

Stott~\cite{Stott2017dis}, \cf \cite{GaubertStott2017},
\cite{Stott2016arxiv}, has thoroughly studied the set $\frakL$ if
$\frakM$ is a set of two matrices. Part of his results depend on the
indefinite metric generated by the difference of these two
matrices. We give an alternative proof of one of Stott's main results,
which, as it seems to us, points out this dependence in a more lucid
way. We refer to \cite{AzizovIokhvidov1989} or \cite{Bognar1974} for
information on the geometry of spaces with an indefinite metric.

Let $\mathscr H=\mathbb C^{n}$ be an $n$-dimensional Hilbert space and
$\frakM:=\{A,0\}$, where $A$ is an invertible $n\times n$
matrix. Clearly, $\inf\frakM=0$ or $\inf\frakM=A$ if $A$ or $-A$,
resp., are positive. It remains to discuss the case that $A$ has $p$,
$p\in\{1,\ldots,n-1\}$, positive and $q:=n-p$ negative eigenvalues. By
\eqref{eq2.4} it is enough to study the case that $A$ has a $2\times 2$
block matrix representation
\[
A= \begin{pmatrix}
    I_{p} & 0 \\ 0 & -I_{q}
    \end{pmatrix} =: J
\]
according to the orthogonal decomposition $\mathscr H=\mathscr
H_{1}\oplus\mathscr H_{2}$, where $\mathscr H_{1}=\mathbb C^{p}$,
$\mathscr H_{2}=\mathbb C^{q}$ and $I_{p}$ and $I_{q}$ denote the
$p\times p$ and $q\times q$, resp., unit matrices.

Denote by $\mathbb C^{p\times q}$ the linear space of all $p\times q$
matrices with complex entries. For $X\in\mathbb C^{p\times q}$ define
an $n\times n$ matrix $S(X)$ by
\begin{equation}
\label{eq5.3a}
S(X):=
\begin{pmatrix}
  I_{p}+XX^{\ast} & (I_{p}+XX^{\ast})^{\sfrac{1}{2}}X \\
  X^{\ast}(I_{p}+XX^{\ast})^{\sfrac{1}{2}} & X^{\ast}X
\end{pmatrix} .
\end{equation}

\begin{theorem}
\label{thm5.5}
The map
\[
X\mapsto M(X) := J - S(X), \; X\in\mathbb C^{p\times q}
\]
establishes a one-to-one correspondence between
$\mathbb C^{p\times q}$ and $\Max\frakL$. 
\end{theorem}

\begin{proof}
Step 1: For $X\in\mathbb C^{p\times q}$, the matrix $M(X)$ is a
maximal lower bound of $\frakM$. Since $S(X)\in\frakS^{+}$ by
Theorem~\ref{thm2.2}, $M(X)\leq J$. Since
\[
M(X)=
- \begin{pmatrix}
  XX^{\ast} & (I_{p}+XX^{\ast})^{\sfrac{1}{2}}X \\
  X^{\ast}(I_{p}+XX^{\ast})^{\sfrac{1}{2}} & I_{q}+X^{\ast}X
  \end{pmatrix}
\]
and
\[
(I_{p}+XX^{\ast})^{\sfrac{1}{2}}X = X(I_{q}+X^{\ast}X)^{\sfrac{1}{2}}
\]
another application of Theorem~\ref{thm2.2} yields $M(X)\leq0$. To
prove that $M(X)$ is a maximal element of $\frakL$ note first that
\[
  \mathscr N\bigl(S(X)\bigr)=\left\{
    \begin{pmatrix}
    -(I_{p}+XX^{\ast})^{-\sfrac{1}{2}} X x_{2} \\ x_{2}
    \end{pmatrix}
    : x_{2} \in\mathscr H_{2}\right\},
\]  
which yields
\[
\dim\mathscr N\bigl(J-M(X)\bigr) = \dim\mathscr N\bigl(S(X)\bigr)\geq q.
\]
Similarly,
\[
  \mathscr N\bigl(M(X)\bigr)=\left\{
    \begin{pmatrix}
    x_{1} \\
    -(I_{q}+X^{\ast}X)^{-\sfrac{1}{2}} X^{\ast}x_{1} 
    \end{pmatrix}
    : x_{1} \in\mathscr H_{1}\right\},
\]
hence, $\dim\mathscr N\bigl(M(X)\bigr)\geq p$.
Since $M(X)+S(X)=J$ yields
\[
\mathscr N\bigl(M(X)\bigr) \cap \mathscr N\bigl(S(X)\bigr) = \{0\},
\]
we obtain 
\[
\mathscr N\bigl(0-M(X)\bigr) + \mathscr N\bigl(J-M(X)\bigr) = \mathbb C^{n},
\]
and $M(X)\in\Max\frakL$ by \eqref{eq4.3}. \\
Step~2: The correspondence $X\mapsto M(X)$ is one-to-one. Let
$X,Y\in\mathbb C^{p\times q}$ and $M(X)=M(Y)$. Comparing the left
upper  corners of $M(X)$ and $M(Y)$ we get
$XX^{\ast}=YY^{\ast}$. Since $I_{p}+XX^{\ast}$ is invertible, it
follows $X=Y$ from the equality of the right upper corners of $M(X)$
and $M(Y)$. \\
Step~3: For arbitrary $M\in\Max\frakL$, there exists $X\in\mathbb
C^{p\times q}$ satisfying $M=M(X)$. By Lemma~\ref{lem5.4} the space
$\mathscr N(J-M)$ is a negative definite subspace and the space
$\mathscr N(0-M)$ is a positive definite subspace with respect to the
indefinite metric generated by $J$. Since 
\begin{equation}
\label{eq5.4}  
\mathscr N\bigl(0-M\bigr) + \mathscr N\bigl(J-M\bigr) = \mathscr H
\end{equation}
by \eqref{eq4.3}, the space $\mathscr N(0-M)$ is a maximal positive
definite subspace. Let $K$ be its angular operator. Recall that
\[
  \mathscr N(0-M) = \left\{
     \begin{pmatrix} x_{1} \\ Kx_{1} \end{pmatrix}
      : x_{1}\in\mathscr H_{1} \right\}
\]  
and $K$ is a contraction from $\mathscr H_{1}$ into $\mathscr H_{2}$
with the additional property that $I_{p}-K^{\ast}K$ is invertible.
Setting
\[
X:=-(I_{p}-K^{\ast}K)^{-\sfrac{1}{2}}K^{\ast},
\]
we obtain
\begin{align*}
(I_{p}+XX^{\ast}) & =
  I_{p}+(I_{p}-K^{\ast}K)^{-\sfrac{1}{2}}K^{\ast}K
  (I_{p}-K^{\ast}K)^{-\sfrac{1}{2}} 
   = I_{p} + (I_{p}-K^{\ast}K)^{-1} K^{\ast}K \\                    
 & = (I_{p}-K^{\ast}K)^{-1}(I_{p}-K^{\ast}K + K^{\ast}K)
   = (I_{p}-K^{\ast}K)^{-1}
\end{align*}  
and 
\[
X^{\ast}X = K(I_{p}-K^{\ast}K)^{-1}K^{\ast},
\]
hence   
\[
  S(X)=
  \begin{pmatrix}
    (I_{p}-K^{\ast}K)^{-1}    & -(I_{p}-K^{\ast}K)^{-1}K^{\ast} \\
    -K (I_{p}-K^{\ast}K)^{-1} & K (I_{p}-K^{\ast}K)^{-1} K^{\ast}
  \end{pmatrix}.
\]
Now it is not hard to see that for 
\[
u = \begin{pmatrix} x_{1} \\ K x_{1} \end{pmatrix}
    \in\mathscr N(0-M),\quad
M(X)u = \bigl(J-S(X)\bigr) \begin{pmatrix} x_{1} \\ K x_{1} \end{pmatrix}=0,
\]  
thus
\begin{equation}
\label{eq5.5} 
  Mu = M(X)u, u\in\mathscr N(0-M).
\end{equation}
If $u\in\mathscr N(J-M)$ and $v\in\mathscr N(M)$, we have
$(Ju,v)=(Mu,v)=(u,Mv)=0$. It follows that $\mathscr N(J-M)$ is a
subspace of the $J$-orthogonal complement of $\mathscr N(M)$. Since
the $J$-orthogonal complement of a maximal positive definite subspace
$\mathscr N(M)$ is a negative definite subspace and since $\mathscr
N(J-M)$ is a maximal negative definite subspace, we can conclude
that $\mathscr N(J-M)$ is the $J$-orthogonal complement of
\[
 \mathscr N(M) =
 \left\{ \begin{pmatrix} x_{1}\\ K x_{1} \end{pmatrix} :
  x_{1}\in\mathscr H_{1} \right\},
\]  
which yields
\[
  \mathscr N(J-M)=
  \left\{ \begin{pmatrix} K^{\ast}x_{2}\\x_{2} \end{pmatrix} :
  x_{2}\in\mathscr H_{2} \right\}.
\]  
It is easy to see that
\[
  M(X)\begin{pmatrix} K^{\ast}x_{2}\\x_{2} \end{pmatrix}
  = \bigl(J-S(X)\bigr)
    \begin{pmatrix} K^{\ast}x_{2}\\x_{2} \end{pmatrix}
  = \begin{pmatrix} K^{\ast}x_{2}\\-x_{2} \end{pmatrix}
\] 
or, equivalently, $M(X)v=Jv$, $v\in\mathscr N(J-M)$. It follows
\begin{equation}
\label{eq5.6}
Mv = M(X)v,\;v\in\mathscr(J-M).  
\end{equation}
Finally, equality $M=M(X)$ is a consequence of \eqref{eq5.4},
\eqref{eq5.5}, and \eqref{eq5.6}.
\end{proof}

\section{Infimum and maximal lower bounds under restrictions}
\label{sec6}
The present section is mainly devoted to a generalization of
Theorem~\ref{thm1.2} and to an extension of Theorem~\ref{thm1.3} if
$\dim\mathscr H<\infty$. For a set $\frakM$, define
\[
  \frakL^{com}=\frakL^{com}(\frakM)
  :=\bigl\{L\in\frakL: AL=LA, A\in\frakM\bigr\}.
\]
If the greatest element $G(\frakL^{com})$ of the set $\frakL^{com}$
exists and if $M$ is a maximal lower bound of $\frakM$ commuting with
all operators of $\frakM$, obviously, $G(\frakL^{com})=M$.
Thus, we can state the following result.

\begin{theorem}
\label{thm6.1}
Let $\frakM$ be such that $G(\frakL^{com})$ exists. There exists at
most one maximal lower bound of $\frakM$ commuting with all operators
of $\frakM$. \hfill\MYQED  
\end{theorem}

\begin{example}
\label{ex6.2}
Let
\[
  \frakM:=\left\{
    \begin{pmatrix}1&0\\0&0\end{pmatrix},
    \begin{pmatrix}1&1\\1&2\end{pmatrix}
   \right\}.
\]
It is easy to see that only scalar matrices commute with both
matrices  of  $\frakM$. It follows $G(\frakL^{com})=0$. Since the matrix
$\begin{psmallmatrix}\frac{1}{2}&0\\0&0 \end{psmallmatrix}$ is a
lower bound of $\frakM$, a maximal lower bound of $\frakM$ commuting
with both matrices of $\frakM$ does not exist.
\end{example}

Deriving a generalization of Theorem~\ref{thm1.2} we first prove that
the greatest element $G(\frakL^{com})$ exists if $\frakM$ is a finite
set. At the same time, our proof delivers an algorithm for computing
$G(\frakL^{com})$ inductively.

Let $\frakM:=\{A_{n}:n\in\mathbb N\}$ be a set of pairwise commuting
operators. We set
\begin{equation}
\label{eq6.1}
M_{1}:= A_{1},\; M_{n+1}:=\bigl(M_{n}+A_{n+1}-|M_{n}-A_{n+1}|\bigr),
\;n\in\mathbb N.
\end{equation}
Let $\frakM_{n}$ denote the subset $\frakM_{n}:=\{A_{k}:
k=1,\ldots,n\}$ of $\frakM$ and
$\widetilde\frakM_{n+1}:=\{M_{n},A_{n+1}\}$,  $n\in\mathbb N$.

\begin{theorem}
\label{thm6.3}
For $n\in\mathbb N$, the operator $M_{n}$ is the greatest element of
$\frakL^{com}(\frakM_{n})$. 
\end{theorem}

\begin{proof}
From \eqref{eq6.1} it can be concluded that $M_{n}$ is a function of
$A_{k}$,
$k=1,\ldots,n$. Since the operators of $\frakM$ commute pairwise, it
follows that $M_{n}$ commutes with all operators of $\frakM$ and
$M_{n+1}$ commutes with $M_{n}$. Moreover, if $S\in\frakS$ commutes
with all operators of $\frakM$, it commutes with $M_{n}$. Let us prove
the first assertion by induction over $n$. The case $n=1$ is trivial,
and the case $n=2$ is covered by Theorem~\ref{thm1.2}. It follows that
\begin{equation}
  \label{eq6.2}
  M_{n+1}=G\bigl(\frakL^{com}(\widetilde\frakM_{n+1})\bigr)
\end{equation}
particularly,
\begin{equation}
  \label{eq6.3}
  M_{n+1}\leq M_{n}.
\end{equation}
Since $M_{n}=G\bigl(\frakL^{com}(\frakM_{n})\bigr)$ by induction assumption,
we have
\[
  M_{n+1}\in \frakL^{com}(\frakM_{n+1}).
\]
If $L\in\frakL^{com}(\frakM_{n+1})$, then
$L\in\frakL^{com}(\frakM_{n})$ and $L$ commutes with $A_{n+1}$. It
follows $L\leq M_{n}$, hence
$L\in\frakL^{com}(\widetilde{\frakM}_{n+1})$ and $L\leq M_{n+1}$ by
\eqref{eq6.2}, which yields
\[
  M_{n+1}=G\bigl(\frakL^{com}(\frakM_{n+1})\bigr). \tag*{\MYQED}
\]
\renewcommand{\qed}{}
\end{proof}

\begin{theorem}
\label{thm6.3a}
If $\frakM$ is a set of pairwise commuting operators, the greatest
element of $\frakL^{com}$ exists. 
\end{theorem}

\begin{proof}
We can presuppose that $\frakM$ is a subset of $\frakS^{+}$. Let
$\Lambda$ denote the set of finite subsets of $\frakM$ partially
ordered by the inclusion relation. For $\lambda\in\Lambda$, let
$M_{\lambda}$ be the greatest element of $\frakL^{com}(\lambda)$,
existing according to Theorem~\ref{thm6.3}. The set $\frakN:=\{-M_{\lambda}:\lambda\in\Lambda\}$ satisfies the
conditions of Theorem~\ref{thm2.1}. Denote the greatest element of
$\frakN$ by $-M$. From \eqref{eq2.1a} and the definition of the
weak-operator topology easily follows that $-M$ is an element of the
weak-operator closure of $\frakN$. Therefore, $-M$ commutes with all
operators of $\frakM$, and we obtain $M\in\frakL^{com}(\frakM)$.
On the other hand, if $L\in\frakL^{com}(\frakM)$, then 
$L\in\frakL^{com}(\lambda)$, hence $L\leq M_{\lambda}$ for all
$\lambda\in\Lambda$, which implies $L\leq M$. It follows that $M$ is
the greatest element of $\frakL^{com}$. 
\end{proof}

From Theorems~\ref{thm6.1} and \ref{thm6.3a} the question arises,
whether for an arbitrary set $\frakM$ of pairwise commuting
operators there exists a maximal lower bound of $\frakM$ commuting
with all operators of $\frakM$. We only have two partial answers.
Our first result generalizes Theorem~6 of \cite{Aslanov2005} from a
set $\frakM$ of two commuting operators to an arbitrary finite set of
pairwise commuting operators. Our second result gives an affirmative
answer to our question if $\dim\mathscr H < \infty$.

\begin{theorem}
\label{thm6.4}
Let $\frakM:=\{A_{1},\ldots,A_{n}\}$ be a finite set of pairwise
commuting operators. There exists a maximal lower bound $M$ of
$\frakM$ commuting with all operators of $\frakM$.  Moreover, for
$j,k\in\{1,\ldots,n\}$, the orthogonal complements of $\mathscr
N(A_{j}-M)\cap\mathscr N(A_{k}-M)$ in the subspaces $\mathscr
N(A_{j}-M)$
and $\mathscr N(A_{k}-M)$, resp., are orthogonal to one another.
\end{theorem}

\begin{proof}
By a theorem of von Neuman, \cf \cite[Nr.~130]{RieszNagy1973}, there
exists an operator $S\in\frakS$ such that all operators of $\frakM$
are functions of $S$. If
\[
  S = \int_{a}^{b} t\,E(dt)
\]
is a spectral representation of $S$, we can write
\[
   A_{k} = \int_{a}^{b} f_{k}(t)\,E(dt), 
\]
where $f_{k}$ is a measurable real function on $[a,b]$,
$k=1,\ldots,n$. Set
\[
  f(t):=\min\bigl\{f_{k}(t): k=1,\ldots,n\bigr\},\;t\in[a,b],
\]
and
\[
  M:=\int_{a}^{b} f(t)\,E(dt).
\]
The operator $M$ is a lower bound of $\frakM$ and commutes with all
operators of $\frakM$. If
\[
  \Delta_{k}:=\bigl\{t\in[a,b] : f(t)=f_{k}(t)\bigr\}
\]
we obtain
\begin{equation}
\label{eq6.4}
\mathscr N(A_{k}-M) = E(\Delta_{k})\mathscr H, \;k=1,\ldots,n  
\end{equation}
which implies
\[
  \sum_{k=1}^{n}\mathscr N(A_{k}-M)=\mathscr H
\]
and $M\in\Max\frakL$ by \eqref{eq4.3}. The second assertion of the
theorem immediately follows from \eqref{eq6.4}.
\end{proof}

\begin{theorem}
\label{thm6.7}
If $\frakM$ is a set of pairwise commuting matrices, there exists a
unique maximal lower bound of $\frakM$ commuting with all matrices of $\frakM$.  
\end{theorem}

\begin{proof}
Since the matrices of $\frakM$ can be diagonalized simultaneously by a
unitary matrix, see \cite[Thm.~2.5.5]{HornJohnson2013}, 
according to \eqref{eq2.4} it is enough to prove the
result for a set $\frakM$ of diagonal $n\times n$ matrices. Let $M$ be
a diagonal matrix, whose $k$-th element on the principal diagonal is
equal to the infimum of the $k$-th elements on the principal diagonal
of the matrices of $\frakM$, $k=1,\ldots,n$. Clearly, $M\in\frakL$ and
$M$ commutes with all matrices of $\frakM$. If $L\in\frakL$ and $M\leq
L$, the $k$-th elements on the principal diagonals of $M$ and $L$
coincide. The relation $L-M\in\frakS^{+}$ implies that $L$ is a diagonal
matrix, hence $L=M$ and $M\in\Max\frakL$. The uniqueness property is a
consequence of Theorem~\ref{thm6.3a} and Theorem~\ref{thm6.1}.
\end{proof}

We conclude the first part of the present section with a remark that
essential assertions on $\frakL^{com}(\frakM)$ seem to be not known if
the operators of $\frakM$ do not commute pairwise. 

Let $\frakL^{+}:=\frakL^{+}(\frakM)$ denote the set
$\frakL^{+}=\frakL\cap\frakS^{+}$ of all positive lower bounds of
$\frakM$. Since $\frakL^{+}$ is not empty if and only if $\frakM$ is a
set of positive operators, to the end of our paper we shall require that
$\frakM\subseteq\frakS^{+}$. It is easy to see that $\frakL^{+}$
has the greatest element $G(\frakL^{+})$ if and only if there exists a
unique positive maximal element $M$ of $\frakL$ and then $G(\frakL^{+})=M$.

\begin{theorem}
\label{thm6.7a}
If
\begin{equation}
  \label{eq6.5}
  \gamma:=\inf\bigl\{(Au,u):A\in\frakM, u\in\mathscr H,\|u\|=1\bigr\}>0,
\end{equation}
the greatest element $G(\frakL)$ of $\frakL$ exists if and only if the
greatest element $G(\frakL^{+})$ of $\frakL^{+}$ exists. In this case, $G(\frakL)=G(\frakL^{+})$.
\end{theorem}

\begin{proof}
If $G(\frakL)$ exists, it is positive, and, hence, also the greatest
element of $\frakL^{+}$. Conversely, let $G(\frakL^{+})$ exist. Since
$\gamma I\in\frakL^{+}$ by \eqref{eq6.5},
\begin{equation}
\label{eq6.6}
G(\frakL^{+})\geq\gamma I.  
\end{equation}
If $G(\frakL)$ would not exist, we could find $L\in\frakL$ and
$v\in\mathscr H$ with
\[
  \bigl(G(\frakL^{+})v,v\bigr) < (Lv,v),
\]
hence, for $\varepsilon\in(0,1)$,
\begin{equation}
\label{eq6.7}
\bigl(\bigl((1-\varepsilon)G(\frakL^{+})+\varepsilon L\bigr)v,v\bigr) >
  (1-\varepsilon)\bigl(G(\frakL^{+})v,v\bigr) +
  \varepsilon\bigl(G(\frakL^{+})v,v\bigr) 
= \bigl(G(\frakL^{+})v,v\bigr).
\end{equation}
Since from \eqref{eq6.6} follows that $(1-
\varepsilon)G(\frakL^{+})+\varepsilon L\in\frakL^{+}$ for $\varepsilon$
small enough, relation \eqref{eq6.7} contradicts the assumption that
$G(\frakL^{+})$ is the greatest element of $\frakL^{+}$.
\end{proof}

The following Theorems~\ref{thm6.8} and \ref{thm6.11} generalize
Proposition~2.4 and Theorem~2.5, resp., of
\cite{GheondeaGudderJonas2005}, stated there for a set $\frakM$ of two operators.

Let $\mathscr K$ denote the subspace 
\begin{equation}
\label{eq6.7a}
\mathscr K:=
\overline{\bigcap_{A\in\frakM}\mathscr R(A^{\sfrac{1}{2}})}.
\end{equation}
Let $\mathscr H_{2}$ be a subspace of $\mathscr H$ satisfying
$\mathscr K\subseteq\mathscr H_{2}$  and $\mathscr H_{1}$ be its orthogonal
complement. Partition $S\in\frakS$ according to \eqref{eq2.5} and
define $S/S_{1}$ and $\frakM/\mathscr H_{1}$ according to
\eqref{eq2.6}  and \eqref{eq2.7}, resp..

\begin{theorem}
\label{thm6.8}
The greatest element $G(\frakL^{+})$ of $\frakL^{+}$ exists if and only
if the greatest element $G\bigl(\frakL^{+}(\frakM/\mathscr H_{1})\bigr)$
of $\frakL^{+}(\frakM/\mathscr H_{1})$ exists. In this case,
$G(\frakL^{+})$ has the form
\[
  G(\frakL^{+}) =
  \begin{pmatrix}
  0 & 0 & \\ 0 & G\bigl(\frakL^{+}(\frakM/\mathscr H_{1})\bigr)
  \end{pmatrix}.
\]
\end{theorem}

\begin{proof}
If $L\in\frakL^{+}$, $\mathscr R(L^{\sfrac{1}{2}})\subseteq\mathscr K$
by Douglas' theorem \cite{Douglas1966}. It follows $\mathscr
H_{1}\subseteq\mathscr N(L)$ and
\begin{equation}
\label{eq6.8}
L= \begin{pmatrix} 0&0\\0&L_{2} \end{pmatrix}  
\end{equation}
for some $L_{2}\in\frakL^{+}(\frakM/\mathscr H_{1})$ by
Theorem~\ref{thm2.2}. Conversely, if $L$  has the form \eqref{eq6.8},
it belongs to $\frakL^{+}$. Thus, the assertion of the theorem follows immediately.
\end{proof}

\begin{corollary}
\label{cor6.9}
If $\dim\mathscr K$ is at most $1$, the greatest element of $\frakL^{+}$
exists. 
\end{corollary}

\begin{proof}
If $\mathscr K=\{0\}$, from the first assertion of Lemma~\ref{lem4.1}
follows $G(\frakL^{+})=0$. If $\dim\mathscr K=1$, choose
$\mathscr H_{2}:=\mathscr K$. Since the elements of
$\frakL^{+}(\frakM/\mathscr H_{1})$ are nonnegative real numbers,
$G\bigl(\frakL^{+}(\frakM/\mathscr H_{1})\bigr)$ exists, and the
result follows from Theorem~\ref{thm6.8}. 
\end{proof}

\begin{corollary}
\label{cor6.10}
Let $A$ be a positive contraction and $P$ be a projection. The
greatest element of $\frakL\bigl(\{A,P\}\bigr)$ exists and is equal to
$A/A_{1}$, where $\mathscr H_{1}$ is the orthogonal complement of
$\mathscr H_{2}=\mathscr R(P)$.
\end{corollary}

\begin{proof}
The assertion follows from Theorem~\ref{thm6.8} and the fact that
$A/A_{1}$ is a positive contraction if $A$ is.
\end{proof}

\begin{corollary}
\label{cor6.10a}
If the infimum of the set $\frakM/\mathscr K^{\perp}$ exists, then the
set $\frakL^{+}(\frakM)$ has the greatest element.
\end{corollary}

\begin{proof}
If the set $\frakM/\mathscr K^{\perp}$ of positive operators has the
infimum or, equivalently, the set $\frakL(\frakM/\mathscr K^{\perp})$  
has the greatest element, it is positive and clearly the greatest
element of $\frakL^{+}(\frakM/\mathscr K^{\perp})$. Now
Theorem~\ref{thm6.8} can be applied.
\end{proof}

\begin{theorem}
\label{thm6.11}
Let $\frakM$ be a finite subset of $\frakS^{+}$. Let $\mathscr H_{1}$
be a subspace of $\mathscr H$ and $\mathscr H_{2}$ be its orthogonal
complement.  If\; $G\bigl(\frakL^{+}(\frakM/\mathscr H_{1})\bigr)$
exists and $G(\frakL^{+})$ is equal to
\begin{equation}
\label{eq6.9}
G(\frakL^{+}) =
   \begin{pmatrix} 0 & 0 \\
        0 & G\bigl(\frakL^{+}(\frakM/\mathscr H_{1})\bigr) 
   \end{pmatrix}, 
\end{equation}
the inclusion $\mathscr K\subseteq\mathscr H_{2}$ is satisfied.
\end{theorem}

\begin{proof}
Let $S$ be the parallel sum of the operators of $\frakM$. Since
$S\in\frakL^{+}$ and
$\overline{\mathscr R(S^{\sfrac{1}{2}})}=\mathscr K$, it
follows $G(\frakL^{+})u\neq0$, $u\in \mathscr
K\setminus\{0\}$. Therefore, if  $G(\frakL^{+})$ has the form
\eqref{eq6.9} according to the orthogonal decomposition $\mathscr
H=\mathscr H_{1}\oplus\mathscr H_{2}$, the space $\mathscr H_{1}$ is a      
subspace of the orthogonal complement of $\mathscr K$ or, what is the      
same, $\mathscr K\subseteq\mathscr H_{2}$.
\end{proof}

Finally, let us discuss generalizations of Ando's theorem, \cf
Theorem~\ref{thm1.3} above. For $A,B\in\frakS^{+}$, denote by $A:B$
its  parallel sum. Following Ando~\cite{Ando1999}, we set
\[
  [A]B:=\lim_{m\to\infty}(mA):B.
\]  

\begin{lemma}
\label{lem6.12}  
If $A,B\in\frakS^{+}$, the following assertions are true:
\begin{enumerate}
\item[(i)]
  $[A]B=B$ if\: 
  $\mathscr R(B^{\sfrac{1}{2}})\subseteq\mathscr R(A^{\sfrac{1}{2}})$,
\item[(ii)]
  $[A:B]B=[A]B$.
\end{enumerate}
\end{lemma}

\begin{proof}
By a theorem of Kosaki~\cite{Kosaki1984} and
Pekarev~\cite{Pekarev1978}, the operator $[A]B$ is equal to
$[A]B=B^{\sfrac{1}{2}}P_{A,B}B^{\sfrac{1}{2}}$, where $P_{A,B}$ is
the projection onto the closure of $\bigl\{u\in\mathscr
H:B^{\sfrac{1}{2}}u\in\mathscr R(A^{\sfrac{1}{2}})\bigr\}$. If
$\mathscr R(B^{\sfrac{1}{2}})\subseteq\mathscr
R(A^{\sfrac{1}{2}})$, the projection $P_{A,B}$ is equal to $I$
and (i) follows. Since
\[
  \mathscr R\bigl((A:B)^{\sfrac{1}{2}}\bigr)=
  \mathscr R(A^{\sfrac{1}{2}})\cap\mathscr R(B^{\sfrac{1}{2}}),
\]
it follows
\[
\bigl\{u\in\mathscr H: B^{\sfrac{1}{2}}u\in
\mathscr R(A^{\sfrac{1}{2}})\bigr\} =
\bigl\{u\in\mathscr H: B^{\sfrac{1}{2}}u\in\mathscr R\bigl((A:B)^{\sfrac{1}{2}}\bigr)\bigr\},
\]  
hence $P_{A,B}=P_{A:B,B}$, which yields (ii).
\end{proof}

\begin{theorem}
\label{thm6.13}
Let $\frakM:=\bigl\{A^{(1)},\ldots,A^{(n)}\bigr\}$ be a finite subset
of $\frakS^{+}$ and $S$ denote the parallel sum of the operators of
$\frakM$. If there exists $k\in\{1,\ldots,n\}$ such that
\begin{equation}
\label{eq6.10}
[S]A^{(k)} \leq [S]A^{(j)},\quad j=1,\ldots,n,  
\end{equation}
the greatest element $G(\frakL^{+})$ of $\frakL^{+}$ exists and is equal
to $[S]A^{(k)}$.
\end{theorem}

\begin{proof}
The operator $[S]A^{(k)}$ is positive and $S\leq A^{(j)}$,
$j=1,\ldots,n$. It follows
\[
[S]A^{(k)}\leq[S]A^{(j)}\leq[A^{(j)}]A^{(j)}=A^{(j)}
\]  
by \eqref{eq6.10}, which yields $[S]A^{(k)}\in\frakL^{+}$. If
$L\in\frakL^{+}$, we have
\[
  \mathscr R(L^{\sfrac{1}{2}})\subseteq \bigcap_{j=1}^{n}
  \mathscr R\bigl(\bigl(A^{(j)}\bigr)^{\sfrac{1}{2}}\bigr)
  = \mathscr R(S^{\sfrac{1}{2}})
\]
hence, $L=[S]L\leq[S]A^{(k)}$ by Lemma~\ref{lem6.12}. It follows that
$[S]A^{(k)}$ is the greatest element of $\frakL^{+}$.
\end{proof}

Recall that by Corollary~\ref{cor3.4} condition \eqref{eq6.10} is
equivalent to the fact that the set
\begin{equation}
\label{eq6.11}
\widetilde{\frakM} := \bigl\{[S]A^{(j)} \colon j=1,\ldots,n\bigr\}
\end{equation}
has the greatest lower bound.

Taking into account Lemma~\ref{lem6.12}~(ii), we obtain that the
preceding theorem is a generalization of the easy half of
Theorem~\ref{thm1.3} from a set $\frakM$ of two operators to an
arbitrary finite set. We conjecture that under the conditions of
Theorem~\ref{thm6.13} the greatest element $G(\frakL^{+})$ of
$\frakL^{+}$ exists if and only if $\widetilde{\frakM}$ has the greatest
lower bound and can prove this if $\dim \mathscr H < \infty$. 

Recall that Ando's theorem is equivalent to the theorem of Moreland
and Gudder \cite{MorelandGudder1999} if $\dim\mathscr H <\infty$. This
follows from Theorems~5.15, 5.18 of \cite{Ando2005}.

\begin{theorem}
\label{thm6.14}
Let $\dim\mathscr H < \infty$ and
$\frakM:=\bigl\{A^{(1)},\ldots,A^{(n)}\bigr\}$ be a finite set of
nonnegative Hermitian matrices. The greatest element $G(\frakL^{+})$ of
$\frakL^{+}$ exists if and only if the set $\widetilde{\frakM}$
of\eqref{eq6.11} has the greatest lower bound.  
\end{theorem}

\begin{proof}
The `if'-part follows from Theorem~\ref{thm6.13}. Conversely, let
$G(\frakL^{+})$ exist. Since $\dim\mathscr H < \infty$, the space 
$\mathscr H_{2}:=\mathscr R(S)=\mathscr R(S^{\sfrac{1}{2}})$ is a
closed subspace of $\mathscr H$. If $\mathscr H_{2}=\{0\}$, we have
$[S]A^{(j)}=[0]A^{(j)}=0$, hence, $\widetilde{\frakM}=\{0\}$ and $0$
is the greatest lower bound of $\widetilde{\frakM}$. If $\mathscr
H_{2}=\mathscr H$, then $\mathscr R(S^{\sfrac{1}{2}})=\mathscr H_{2}$,
$[S]A^{(j)}=A^{(j)}$ by Lemma~\ref{lem6.12}~(i). Therefore,
$\frakM=\widetilde{\frakM}$ and if $G(\frakL^{+})$ exists, then
$G\bigl(\frakL^{+}(\widetilde{\frakM})\bigr)$ exists as well. Since
$\mathscr R$$\bigl(A^{(j)}\bigr)\allowbreak=\mathscr H$, condition \eqref{eq6.5} is
satisfied and $G\bigl(\frakL(\widetilde{\frakM})\bigr)$ exists by
Theorem~\ref{thm6.7}. Finally, let $\mathscr H_{2}$ be a proper
subspace of $\mathscr H$ and let $\mathscr H_{1}$ be its orthogonal
complement. By Theorem~5.18 of \cite{Ando2005},
\[
  [S]A^{((j)} =
  \begin{pmatrix}
  0 & 0 \\ 0 & A^{(j)}/A_{1}^{(j)}
   \end{pmatrix} 
\]
according to $\mathscr H = \mathscr H_{1}\oplus\mathscr H_{2}$. From
Theorem~\ref{thm6.8} follows that
\[
  G\bigl(\frakL^{+}(\widetilde{\frakM})\bigr) =
  \begin{pmatrix}
  0 & 0 \\ 0 & G(\frakL^{+}(\frakM/\mathscr H_{1}))
  \end{pmatrix} 
\]
exists if $G(\frakL^{+})$ exists. Since
$\mathscr R\bigl(A^{(j)}/A^{(j)}_{1}\bigr)=\mathscr H_{2}$ by
Theorem~5.6 of \cite{Ando2005}, Theorem~\ref{thm6.7} can be applied to
the set $\frakM/\mathscr H_{1}$. It follows that the greatest lower
bound of  $\frakM/\mathscr H_{1}$ exists. From Corollary~\ref{cor4.3a}
and Theorem~\ref{thm2.2} it can be concluded that the greatest lower
bound  of $\widetilde{\frakM}$ exists as well.
\end{proof}

Note that the space $\mathscr H_{2}$ of the preceding proof is the
same as the space $\mathscr K$ of \eqref{eq6.8}. In particular,
Theorem~\ref{thm6.14} says that if $\frakM$ is a finite set of
positive invertible matrices, then $\inf\frakM$ exists if and only if
$\frakL^{+}(\frakM)$ has the greatest element. The set
\[
  \frakM =\left\{
    \begin{mypmatrix}
      1+\frac{1}{n} & \frac{1}{\sqrt{n}} \\
      \frac{1}{\sqrt{n}} & \frac{1}{n}
    \end{mypmatrix}
    ; n\in\mathbb N\right\}
\]  
of positive invertible $2\times 2$ matrices shows that such an
assertion is not true if $\frakM$ is countable, \cf Example~\ref{ex4.3}.

\vspace{0pt plus 1fill} % Fügt unsichtbaren, unendlich dehnbaren Kleber ein

%% References

%% bei benutzung von biblatex
% \printbibliography


%% bibtex style file
%% \bibliographystyle{plain}
%% \bibliographystyle{amsplain}
%% \bibliographystyle{amsplainalpha}
%% \bibliographystyle{emss}
\bibliographystyle{arxiv}

% \nocite{*}
% used bibtex database file
\bibliography{minimum}






% \end{thebibliography}

% \input{"minimumVer2.bbl"}

\end{document}